\documentclass{amsart}

\usepackage{amsmath,amssymb,amsthm}
\usepackage{graphicx}
\usepackage{pdfpages}
\usepackage{verbatim}
\usepackage{tikz}
\usepackage{enumitem}
\usetikzlibrary{arrows,positioning}
\tikzset{
    tile/.style={
           circle,
           draw=black},
}
\usepackage{diagbox}
\usepackage{mathdots}
\usepackage{xcolor}

\newcommand{\uhr}{\upharpoonright}

\newtheorem{definition}{Definition}
\newtheorem{theorem}{Theorem}
\newtheorem{prop}{Proposition}
\newtheorem{lemma}{Lemma}
\newtheorem{question}{Question}
\newtheorem{corollary}{Corollary}

\newcommand{\bdefn}{\begin{defn}}
\newcommand{\edefn}{\end{defn}}
\newcommand{\bthm}{\begin{thm}}
\newcommand{\ethm}{\end{thm}}
\newcommand{\bitem}{\begin{itemize}}
\newcommand{\eitem}{\end{itemize}}
\newcommand{\bpf}{\begin{proof}}
\newcommand{\epf}{\end{proof}}

\newcommand{\on}{\operatorname}
\newcommand{\poly}{\on{poly}}
\newcommand{\Nat}{\mathbb{N}}

\newcommand{\Z}{\mathbb{Z}}

\newcommand{\type}{\operatorname{type}}
\newcommand{\alignment}{\operatorname{alignment}}

\newcommand{\concat}{^\smallfrown}

\newcommand{\lex}{\leq_{\text{lex}}}
\newcommand{\tlex}{\text{lex}}
\newcommand{\kO}{\mathcal O}

\newcommand{\gray}[1]{\colorbox{gray!30} {$#1$} }
\newcommand{\gleftarrow}{\gray \leftarrow}
\newcommand{\grightarrow}{\gray \rightarrow}
\newcommand{\guparrow}{\gray \uparrow}
\newcommand{\gdownarrow}{\gray \downarrow}
\newcommand{\gulcorner}{\gray \ulcorner}
\newcommand{\gllcorner}{\gray \llcorner}
\newcommand{\glrcorner}{\gray \lrcorner}
\newcommand{\gurcorner}{\gray \urcorner}

\title{Topological completely positive entropy is no simpler in $\mathbb Z^2$-SFTs}

\author[L.B.~Westrick]{Linda Westrick} 
\address{Penn State University Department of Mathematics, University Park, PA}
 \email{westrick@psu.edu}

\thanks{The author was supported by the Cada R. and Susan Wynn Grove Early Career Professorship in Mathematics.}

\begin{document}

\begin{abstract}
We construct $\mathbb Z^2$-SFTs at every computable level of 
the hierarchy of topological completely positive entropy (TCPE), answering 
 Barbieri and Garc\'ia-Ramos, who asked if there was one 
 at level 3.  Furthermore, 
we show the property of TCPE in $\mathbb Z^2$-SFTs is coanalytic
complete.  Thus there is no simpler description of TCPE 
in $\mathbb Z^2$-SFTs than in the general case.  \end{abstract}

\maketitle

\section{Introduction}

Despite their similar definitions, the shifts of finite type (SFTs) over $\mathbb Z$ 
and the SFTs over $\mathbb Z^2$ often display very different properties.  
For example, there is an algorithm to determine whether a $\mathbb Z$-SFT 
is empty, but the corresponding problem for $\mathbb Z^2$-SFTs is 
undecidable \cite{Berger1966}.  Similarly, the possible entropies achievable 
by a $\mathbb Z$-SFT have an algebraic characterization 
(see e.g. \cite[Chapter 4]{LindMarcus1995}), but for a $\mathbb Z^2$-SFT, the 
possible entropies are exactly the non-negative numbers obtainable as 
the limit of \emph{computable} decreasing sequences of rationals 
\cite{HochmanMeyerovitch2010}.
The appearance of computation in both cases is explained in part by the original 
insight of Wang \cite{Wang1962} that an arbitrary Turing computation can be forced 
to appear in any symbolic tiling of the plane that obeys a precisely 
crafted finite set of local restrictions.  

The  
paper of Hochman and Meyerovitch revived the idea that 
superimposing these computations on an existing SFT allows the 
computations to ``read'' what is written on the existing configurations,
and eliminate any configurations which the algorithm deems 
unsatisfactory.  The effect is to forbid more (even infinitely many) 
patterns from the original SFT, at the cost of covering it
with computation graffiti.  It is not well-understood which 
classes of additional patterns can be forbidden in this way; Durand, Levin and Shen
\cite{DurandLevinShen2008} have found complexity-related restrictions. 
Also, the computation infrastructure has the potential to 
modify more properties of the original SFT beyond forbidding 
words.  For example, in \cite{HochmanMeyerovitch2010}, since 
they wanted to control entropy, they needed the computation to
contribute zero entropy to the SFT.  Thus while 
$\mathbb Z^2$-SFTs can demonstrate universal
behavior in some cases, it is not at all obvious when they will do so.

In a recent series of papers, Pavlov \cite{Pavlov2014,Pavlov2018} 
and Barbieri and Garc\'ia-Ramos \cite{BGR-TA} explored the 
property of topological completely positive entropy (TCPE) in 
$\mathbb Z^d$-SFTs.   Defined by 
Blanchard \cite{Blanchard1992} as a topological analog of the 
$K$-property for measurable dynamical systems, TCPE 
seems rather poorly behaved compared to the $K$-property.  
However, Pavlov was able to give a simple characterization of 
TCPE for $\mathbb Z$-SFTs.  The question then remained: will 
there also be a simple characterization for $\mathbb Z^d$-SFTs? 
In this paper we give a negative answer; TCPE is just as
complex in $\mathbb Z^2$-SFTs as it is in general topological 
dynamical systems.

By definition, a topological dynamical system has TCPE if 
all of its nontrivial topological factors have positive entropy.
Although we have not formally introduced these terms, it should 
be clear that this definition involves a quantification over infinite 
objects (the nontrivial topological factors).  Contrast this 
with, for example, the definition of a Cauchy sequence of real numbers: 
$(x_n)$ is Cauchy if for every rational $\varepsilon>0$ there 
is a natural number $N$ such that etc.  
Here all the quantifications involve finite objects.  A property 
that can be expressed using only quantifications 
over finite objects is called \emph{arithmetic} and such 
properties are typically easier to work with 
than properties which require a quantification over infinite 
objects in their definition.  

Roughly speaking (see the Preliminaries for precise definitions), 
a property is coanalytic, or $\Pi^1_1$, 
if it can be expressed using a
single universal quantification over infinite objects\footnote{technically, 
quantification over elements of a Polish space} and any amount 
of quantification over finite objects.  A property is $\Pi^1_1$-complete if it is 
universal among $\Pi^1_1$-properties.  If a property is $\Pi^1_1$-complete,
it has no arithmetic equivalent description.

In \cite{Pavlov2014} and \cite{Pavlov2018}, Pavlov introduced two 
arithmetic properties that a $\mathbb Z^d$-SFT $X$ could have. 
One of them, ZTCPE, he showed was strictly weaker than TCPE.
The other property implied TCPE, and he asked if it provided 
a characterization.  This question was answered in the negative 
by Barbieri and Garc\'ia-Ramos \cite{BGR-TA}, who constructed 
an explicit counterexample with $d=3$.  We generalize 
both of these results by showing that there is no arithmetic 
property that characterizes TCPE in the $\mathbb Z^2$-SFTs.

\begin{theorem}
The property of TCPE is $\Pi^1_1$-complete in the set of $\mathbb Z^2$-SFTs.
\end{theorem}

In the course of proving their result, Barbieri and Garc\'ia-Ramos defined 
an $\omega_1$-length hierarchy within TCPE 
which stratified TCPE into subclasses.  Their explicit counterexample was a 
$\mathbb Z^3$-SFT at level 3 of this hierarchy, and asked 
whether there could be a $\mathbb Z^2$-SFT at level 3.  
We answer this question in a quite general way.
Standard methods of effective descriptive set theory
imply that the TCPE rank of any $\mathbb Z^2$-SFT must 
be a \emph{computable ordinal}, that is, a countable ordinal $\alpha$ for which 
there is a computable linear ordering $R\subseteq \omega\times \omega$ 
whose order type is $\alpha$.  We show that this is the only restriction.
Below, $\omega_1^{ck}$ denotes the supremum of the computable ordinals.

\begin{theorem}\label{thm:2}
For any ordinal $\alpha<\omega_1^{ck}$, there is a 
$\mathbb Z^2$-SFT with TCPE rank $\alpha$.
\end{theorem}

In fact, these two main theorems are closely related; a $\Pi^1_1$ set 
can always be decomposed into an ordinal hierarchy
of simpler subclasses.  Such hierarchy is called a $\Pi^1_1$ rank 
and a standard reference on the topic is \cite{kechris-book}.
Frequently, when all subclasses are populated, the 
same methods used to populate 
the hierarchy yield a proof of $\Pi^1_1$-completeness.  That 
has happened in this case.

Silvere Gangloff has kindly let us know of his progress on this problem:
he has independently constructed $\mathbb Z^2$-SFTs of rank $\alpha$ for each 
$\alpha < \omega^2$ by a different method \cite{Gangloff-PC}.  We have also learned
Ville Salo has recently constructed $\mathbb Z$-subshifts
of all TCPE ranks \cite{Salo-TA}, and conjectured our Theorem \ref{thm:2}.  
The author would also like to thank Sebasti\'an Barbieri
for interesting discussions on this topic.

\section{Preliminaries}

\subsection{Subshifts and TCPE}

Let $\Delta$ denote a finite alphabet. 
A $\mathbb Z^d$-\emph{subshift} is a subset of $\Delta^{\mathbb Z^d}$ that 
is topologically closed (in the product topology, where $\Delta$ has 
the discrete topology) and closed under the $d$-many shift operations 
and their inverses.  An element of $\Delta^S$ is also called a
\emph{configuration}.
A \emph{pattern} is an element of $\Delta^S$
where $S \subseteq \mathbb Z^d$ is a finite subset.  
A pattern $w$ \emph{appears in} a configuration $x$ if there is 
some $g\in \mathbb Z^d$ such that $w = x \uhr g^{-1}(S)$.  
A pair of patterns $w$ and $v$ \emph{coexist} in $x$ if they 
both appear in $x$.  If $g \in \mathbb Z^d$, and $w$ is 
a pattern or $x$ a configuration, let $gw$ and $gx$ denote
the corresponding shifted versions of $w$ and $x$, 
that is, $gx(h) = x(g^{-1}h)$ and $gw(h) = w(g^{-1}h)$.

A subshift $X \subseteq \Delta^{\mathbb Z^d}$ is completely characterized 
by the set of patterns which do not appear in any configuration of $X$. 
Conversely, for any set $F$ of patterns, the set
$$W_F := \{x \in \Delta^{\mathbb Z^d} : \text{ for all } w \in F, w \text{ does not appear in } x\}$$
is a subshift.  A subshift $X$ is called a \emph{shift of finite type} if 
$X = W_F$ for some finite set $F$ of forbidden patterns.

More generally, a \emph{$\mathbb Z^d$-topological dynamical system} (TDS)
is a pair $(X,T)$ where $X$ is compact separable metric space and $T$ is an action 
of $\mathbb Z^d$ on $X$ by homeomorphisms.  A $\mathbb Z^d$-subshift
is a special case of this.  If $(X,T)$ and $(Y,S)$ are two $\mathbb Z^d$-TDS, 
we say that $(Y,S)$ is a \emph{factor} of $(X,T)$ if there is a continuous 
onto function $f:X\rightarrow Y$ such that $Sf=fT$.  
A \emph{sofic} $\mathbb Z^d$-subshift is a $\mathbb Z^d$-subshift which 
is a factor of a $\mathbb Z^d$-SFT.

A TDS has 
\emph{topological completely positive entropy} if all of its nontrivial 
factors have positive entropy.  The unavoidable 
trivial factor is one where $Y$ consists of a single element only.  

For the purposes of 
this paper, we work almost entirely with an equivalent characterization of TCPE due to 
Blanchard \cite{Blanchard1993}.  This characterization makes use of Blanchard's 
local entropy theory and so a few definitions will be required.

If $\mathcal U$ is an open cover of a subshift
$X$, let $\mathcal U_n$ denote the open cover of $X$ which is the 
common refinement of the shifted covers $g^{-1}\mathcal U$ for $g \in [0,n)^d$,
\[
U_n = \bigvee_{g \in [0,n)^d} g^{-1}\mathcal U.
\]
Let $\mathcal N(\mathcal U_n)$ denote the smallest cardinality 
of a subcover of $\mathcal U_n$.  Then $\log \mathcal N(\mathcal U_n)$ 
can be thought of as the 
minimum number of bits needed to 
communicate, for each $x\in X$ and $g \in [0,n)^d$, an element of 
$\mathcal U$ containing $g^{-1}x$.
The \emph{topological entropy} of $X$ relative to $\mathcal U$ is
\[
h(X,\mathcal U) = \lim_{n\rightarrow \infty} \frac{\log\mathcal N(\mathcal U_n)}{n^d}
\]
A pair of elements $x,y \in X$ are an \emph{entropy pair} if 
$h(X, \{K_x^c, K_y^c\}) > 0$ for every disjoint pair of closed sets $K_x,K_y$ 
containing $x$ and $y$ respectively, where $K^c$ denotes the 
complement of $K$.  Blanchard \cite{Blanchard1993} proved the following theorem for 
$\mathbb Z$-topological dynamical systems.  (It also holds in the more general 
context of a $G$-topological dynamical system, where $G$ is a countable 
amenable group, but we do not need the more abstract formulation; 
it is stated in \cite[Theorem 2.4]{BGR-TA}.)  Here is the 
version we need.

\begin{theorem}[Blanchard]
A subshift $X$ has topological completely positive entropy if and only if 
the smallest closed equivalence relation on $X$ containing the entropy
pairs is all of $X^2$.
\end{theorem}

A useful sufficient condition for the entropy pairhood of $x$ and $y$
 is the following.  Two patterns $w,v\in \Delta^S$ 
are \emph{independent} if there is a positive density subset 
$J\subseteq \mathbb Z^d$ such that for all $I \subseteq J$, there is 
some configuration $x \in X$ such that $x\uhr g^{-1}S = w$ for all 
$g \in I$ and $x \uhr g^{-1} = v$ for all $g\in J \setminus I$.  In English, 
$w$ and $v$ are independent if there is a positive density of locations 
such if we place $w$ or $v$ at each of those locations (free choice), 
regardless of our choices it is always possible to fill in the remaining 
symbols to get a valid configuration $x \in X$.  Observe that if $x\uhr S$ 
and $y\uhr S$ are independent patterns for all finite $S \subseteq \mathbb Z^d$, 
then $x$ and $y$ are an entropy-or-equal pair.

Barbieri and Garc\'ia-Ramos \cite{BGR-TA} defined the following hierarchy of 
closed relations and equivalence relations on $X$.  
They first define the set of \emph{entropy-or-equal pairs},  
$$E_1 = \{(x,y) \in X^2 : x=y \text{ or } (x,y) \text{ is an entropy pair}\},$$
and note that this set is closed.  At successor stages, define 
$$E_{\alpha+1} = \begin{cases} 
\text{ the topological closure of } E_\alpha & \text{ if $E_\alpha$ is not closed}\\
\text{ the transitive, symmetric closure of } E_\alpha & \text{ if $E_\alpha$ is not an equiv. rel'n}\\
E_\alpha & \text{ if $E_\alpha$ is a closed equiv. rel'n}\end{cases}$$
At limit stages, $E_\lambda = \cup_{\alpha < \lambda} E_\alpha$.  
They show that $X$ has TCPE if and only if $E_\alpha = X^2$ for some $\alpha$, 
and in this case they
define the \emph{TCPE rank} of $X$ to be the least $\alpha$ at which 
this occurs.  They construct a $\mathbb Z^3$-SFT of TCPE 
rank 3, and they ask whether this can be improved to a $\mathbb Z^2$-SFT.

\begin{question}[Barbieri \& Garc\'ia-Ramos]
Is there a $\mathbb Z^2$-SFT of TCPE rank 3?
\end{question}

Our Theorem \ref{thm:2} answers this question positively and 
then characterizes those TCPE ranks obtainable by 
$\mathbb Z^2$-SFTs to be exactly the computable ordinals.

\subsection{Effective descriptive set theory}
With the exception of Proposition \ref{prop:overflow} below,
we have attempted to make the paper self-contained with respect 
to effective descriptive set theory.  However, we have surely 
not succeeded completely in this, so we also refer the reader to the 
books \cite{kechris-book} on descriptive set theory and
\cite{sacks} on its computable (effective) aspects.

Let $\omega^\omega$ denote the space of all infinite sequences of 
natural numbers, with the product topology.  Most mathematical objects 
can be described or encoded in a natural way by elements of $\omega^\omega$.
A set $A \subseteq \omega^\omega$ is $\Pi^0_n$ if there is a 
computable predicate $P$ such that for all $x \in \omega^\omega$,
\[
x \in A \iff \forall m_1 \exists m_2 \dots Q m_n P(x,m_1,\dots,m_n)
\]
where each $m_i \in \omega$ (or in a set whose members are coded by 
elements of $\omega$) 
and $Q$ is either $\forall$ or $\exists$, depending on the parity of $n$.
For example, the set of all convergent sequences of rational numbers is $\Pi^0_3$
\[
(q_n)_{n\in\omega} \text{ converges } \iff (\forall \epsilon \in \mathbb Q) (\exists N) (\forall n, m )[n,m>N \implies |q_n-q_m| \leq \epsilon].
\]
A set $A\subseteq \omega^\omega$ is \emph{arithmetic} if it is $\Pi^0_n$ for some $n$.
A set $A \subseteq \omega^\omega$ is \emph{coanalytic}, or $\Pi^1_1$, if there is an 
arithmetic predicate $P$ such that for all $x \in \omega^\omega$,
\[
x\in A \iff (\forall y \in \omega^\omega)P(x,y)
\]
For example the property of TCPE is $\Pi^1_1$.  Let $K(X^2)$ denote the closed subsets of $X^2$
with the Hausdorff metric (appropriately encoded as a subset of $\omega^\omega$).  Then
\begin{multline*}
X \text{ has TCPE } \iff (\forall E \in K(X^2))[(\text{$E$ is an equivalence relation and}\\
\text{$E$ contains the entropy-or-equal pairs}) \implies E = X^2]
\end{multline*}

A \emph{tree} $T \subseteq \omega^{<\omega}$ is any set closed under taking initial segments.
For $\sigma, \tau \in \omega^{<\omega}$, we write $\sigma \prec \tau$ to indicate 
that $\sigma$ is a strict initial segment of $\tau$.  A string $\sigma \in T$ is called a \emph{leaf}
if there is no $\tau \in T$ with $\sigma \prec \tau$.  The empty string is denoted $\lambda$.
A \emph{path} through a tree $T$ is an infinite sequence $\rho \in \omega^\omega$, all 
of whose initial segments are in $T$.  The set of paths through $T$ is denoted $[T]$.
A tree $T$ is \emph{well-founded} if $[T]=\emptyset$.  Let $WF$ denote the set of 
well-founded trees, which is $\Pi^1_1$.
\[
T \in WF \iff \forall \rho \in \omega^\omega [\rho \text{ has some initial segment not in } T]
\]
A $\Pi^1_1$ set $A$ is called \emph{$\Pi^1_1$-complete}
if for every other $\Pi^1_1$ set $B$, there is a computable function $f$ such that for all $x \in \omega^\omega$,
\[
x \in B \iff f(x) \in A
\]
The set $WF$ is $\Pi^1_1$-complete.  No $\Pi^1_1$-complete set is arithmetic.  

So far we have discussed only the descriptive complexity of subsets of $\omega^\omega$. 
There is a miniature version of this theory for subsets of $\omega$ (and by extension, 
subsets of any collection of finitely-describable objects, such as SFTs).  A set 
$A \subseteq \omega$ is $\Pi^0_n$, arithmetic, $\Pi^1_1$, or $\Pi^1_1$-complete 
exactly when the same definitions written above are satisfied, with the only change 
being that the elements $x$ whose $A$-membership is being considered are now 
drawn from $\omega$ rather than $\omega^\omega$, and also $B \subseteq \omega$ 
in the definition of $\Pi^1_1$-complete.  No $\Pi^1_1$-complete subset of $\omega$ 
can be arithmetic either.

A tree $T \subseteq \omega^{<\omega}$ 
is computable if there is an algorithm which, given input $\sigma \in \omega^{<\omega}$, 
outputs 1 if $\sigma \in T$ and 0 otherwise.  An \emph{index} for a computable 
tree $T$ is a number $e \in \omega$ such that the $e$th algorithm in some canonical 
list computes $T$ in the sense described above.  The set of indices of computable 
well-founded trees is a canonical $\Pi^1_1$-complete subset of $\omega$.

On every $\Pi^1_1$ set, it is possible to define a \emph{$\Pi^1_1$ rank}, 
a function which maps elements 
of the set to an ordinal rank $<\omega_1$ in a uniform manner 
(for details see \cite{kechris-book}).  A natural $\Pi^1_1$ rank on well-founded treess 
$T$ is defined by induction as follows.  The rank of a leaf $\sigma \in T$ 
is $r_T(\sigma) = 1$.  For any non-leaf $\sigma \in T$, the rank of $\sigma$ is 
$r_T(\sigma) = \sup_{\tau \in T : \sigma \prec \tau} (r_T(\tau) +1)$.  The 
rank of $T$ is $r(T) = r_T(\lambda)$.
Colloquially, the rank of a well-founded tree 
is the ordinal 
number of leaf-removal operations needed to remove the entire tree.  

If $A\subseteq \omega^\omega$ is $\Pi^1_1$-complete, then for any $\Pi^1_1$
rank on $A$, the ranks of elements of $A$ are cofinal (unboundedly large) 
below $\omega_1$.
If $A \subseteq \omega$ is $\Pi^1_1$-complete, then since $A$ is countable, 
there must be some countable upper limit on the ranks of elements of $A$.
A countable ordinal $\alpha$ is \emph{computable} if 
there is a computable linear ordering $R\subseteq \omega\times \omega$ 
whose order type is $\alpha$.  The computable ordinals are also exactly those 
ordinals which can be the rank of a computable well-founded tree.
The supremum of all computable ordinals 
is denoted $\omega_1^{ck}$.  
If $A \subseteq \omega$ is $\Pi^1_1$-complete, 
then for any $\Pi^1_1$ rank on $A$,
the ranks of elements of $A$ are cofinal in $\omega_1^{ck}$.

Heuristically, a sort of converse holds.  If one can show that all 
countable (resp. computable) levels of a $\Pi^1_1$ hierarchy on a 
subset of $\omega^\omega$ (resp. $\omega$) are 
populated, typically one also has the tools to show that the set in 
question is $\Pi^1_1$-complete.

Barbieri and Garc\'ia-Ramos found topological dynamical systems 
at every level of the TCPE hierarchy, giving strong evidence for 
the following theorem (which will also be a side consequence of our 
methods).

\begin{theorem}
The set of $\mathbb Z^2$-TDS with TCPE is 
$\Pi^1_1$-complete, and the TCPE rank is a $\Pi^1_1$ rank on this set.
\end{theorem}
Here the arbitrary $\mathbb Z^d$-TDS are appropriately encoded 
using elements of $\omega^\omega$.  

Our main goal is to show that the situation is no simpler in $\mathbb Z^d$-SFTs, 
which are appropriately encoded using elements of $\omega$.

\begin{theorem}
The set of $\mathbb Z^2$-SFTs with TCPE is 
$\Pi^1_1$-complete, and the TCPE rank is a $\Pi^1_1$ rank on this set.
\end{theorem}

The first step to proving that the TCPE rank is a $\Pi^1_1$ rank is to 
show that every $\mathbb Z^2$-SFT which has TCPE has a computable 
ordinal rank.  This proof is standard but does assume more familiarity with effective 
descriptive set theory than what was outlined in this introduction. 
The standard reference is \cite{sacks}.

\begin{prop}\label{prop:overflow}
If a $\mathbb Z^d$-TDS $(X,T)$ has TCPE, its TCPE rank is less than $\omega_1^{(X,T)}$.
\end{prop}
\begin{proof} To reduce clutter we prove the theorem for computable $(X,T)$;
the reader can check that the proof relativizes.
Recall that $E_\alpha$ is closed whenever $\alpha$ is odd.  If $L$ is a
computable well-order on $\omega\times \omega$ with least element $1$, 
we say that $Y$ is an $E$-hierarchy on $L$ if 
\begin{itemize}
\item $Y^{[1]}$ codes 
$E_1$ as a closed set, 
\item If $b<_Lc$ are successors in $L$, then $Y^{[c]}$ encodes the topological closure of the 
transitive/symmetric closure of $Y^{[b]}$ and
\item If $c$ is a limit in $L$ then $Y^{[c]}$ encodes the topological closure of the union of 
the sets coded by $Y^{[b]}$ for all $b<_L c$.
\end{itemize}  The definition of $Y^{[c]}$ from $\{Y^{[b]} : b <_L c\}$ is arithmetic 
and the definition of an $E$-hierarchy overflows to computable pseudo-wellorders.

Suppose for the sake of contradiction that the TCPE rank of $X$ is at least 
$\omega_1^{ck}$.    If $E_{\omega_1^{ck}+1} = X^2$,
we would have the following 
$\Sigma^1_1$ definition of $\mathcal O$.
\begin{multline*}
a \in \mathcal O \iff a \in {\kO^\ast} \text{ and }\\\exists Y( Y \text{ is an $E$-hierarchy on $\{b : b \leq_\kO a\}$ and $Y^{[a]} \neq X^2$})
\end{multline*}
This is a contradiction since $\kO$ is $\Pi^1_1$-complete.  Similarly, 
if $E_{\omega_1^{ck}+1} \neq X^2$, then since $X$ has TCPE, 
the next closed set $E_{\omega_1^{ck}+3}$ is strictly larger than $E_{\omega_1^{ck}+1}$.
Let $U \subseteq X^2$ be a basic open set such that $E_{\omega_1^{ck}+3}\cap U \neq \emptyset$
but $E_{\omega_1^{ck}+1} \cap U = \emptyset$.  In this case we could also define $\kO$ by
\begin{multline*}
a \in \mathcal O \iff a \in {\kO^\ast} \text{ and }\\\exists Y( Y \text{ is an $E$-hierarchy on $\{b : b \leq_\kO a\}$ and $Y^{[a]} \cap U = \emptyset$})
\end{multline*}
This provides a $\Sigma^1_1$ definition of $\kO$, for if $a^\ast \in \kO^\ast \setminus \kO$, 
there is some $b^\ast \in \kO^\ast \setminus \kO$ with $b^\ast <_\kO a^\ast$.  
Then $E_{\omega_1^{ck}+1}$ is a subset of $Y^{[b^\ast]}$, so 
$E_{\omega_1^{ck}+3}$ is a subset of $Y^{[a^\ast]}$, and thus 
$Y^{[a^\ast]} \cap U \neq \emptyset$.  Again, contradiction.  Therefore, the TCPE 
rank of $X$ is less than $\omega_1^{ck}$.
\end{proof}

\begin{corollary}
If a $\mathbb Z^d$-SFT has TCPE, then its TCPE rank is a computable ordinal.
\end{corollary}

\begin{corollary}
The TCPE rank is a $\Pi^1_1$-rank on the set of $\mathbb Z^d$-TDS 
and the set of $\mathbb Z^d$-SFTs.
\end{corollary}
\begin{proof}
If $X_1$ and $X_2$ are $\mathbb Z^d$-TDS or $\mathbb Z^d$-SFTs and $X_2$ has TCPE, 
then the following are equivalent:
\begin{enumerate}
\item The TCPE rank of $X_1$ is less than or equal to the TCPE rank of $X_2$.
\item There is an $a \in (\kO^\ast)^{X_2}$ and $E$-hierarchies $Y_1$ and $Y_2$
(for $X_1$ and $X_2$ respectively) on 
$a$ such that $Y_1^{[a]} = X_1^2$
and $Y_2^{[b]} \neq X_2^2$ for any $b <_\kO^{X_2} a$.
\item For all $a \in (\kO^\ast)^{X_2}$ and all $E$-hierarchies $Y_1$ and $Y_2$ on $a$, 
if $Y_1^{[a]} = X_1^2$ then $Y_2^{[a]} = X_2^2$.
\end{enumerate}
Proposition \ref{prop:overflow} guarantees that it is safe to use $\kO^\ast$ 
in two places where we wanted to use $\kO$, but could not.
\end{proof}

\subsection{SFT computation}  This section introduces the main technical 
tool used in this paper, the tiling-based SFT computation framework of 
Durand, Romashchenko and Shen \cite{DurandRomashchenkoShen2012}.
A more motivated and 
detailed description of that framework can be found in their 
original paper.  We also mention that \cite[Chapter 3, Chapter 7.1-2]{SipserITC}
contain a good technical introduction to Turing machines and polynomial 
time complexity, and we assume the reader is fluent in this topic.
Here we give a general overview of the ideas 
and terminology of the DRS construction, 
followed by a more technical description of 
their basic module, which will serve as the basis for our constructions.  

A \emph{Wang tile} is a square with colored sides.  Two Wang tiles 
may be placed next to each other if they have the same color on the side 
that they share.  We do not rotate the tiles.  A \emph{tileset} is a finite 
collection of Wang tiles.  Given a finite tileset, the collection of infinite tilings of 
the plane which can be made
with that tileset is a $\mathbb Z^2$-SFT.  From here on we refer to 
infinite tilings of the plane as configurations.
Wang \cite{Wang1962} described a method for turning any Turing machine 
into a tileset such that any configuration which contains a special 
\emph{anchor tile} is also forced to contain a literal picture of
the space-time diagram of an infinite run of the Turing machine.  If the Turing 
machine runs forever, the tiling can go on forever; if the Turing machine 
halts, there is no configuration because there is no way to continue the tiling.  

The anchor tile contains the 
head of the Turing machine and the start of the tape.  If we would like 
to force computations to appear in every configuration, we must require 
the anchor tile to appear in every configuration.  By compactness,
the only way to do this in a subshift is to require the anchor tile to appear 
with positive density.  This means that many computations go on simultaneously.
It is a technical challenge to organize the infinitely many computations 
so they do not interfere with each other, and to guarantee that the algorithm 
gets enough time to run.  This challenge was first solved by Berger \cite{Berger1966} 
with an intricate fractal construction that was subsequently simplified 
by Robinson \cite{Robinson1971}.  Several other solutions have occurred 
over the years, including the one in \cite{DurandRomashchenkoShen2012}
which is used in this paper.

In the DRS system, tiles use a \emph{location part} of their colors to arrange 
themselves into an $N\times N$ grid pattern for some large $N$.  
Central to each $N\times N$ region, there is a \emph{computation zone};
tiles in this zone must participate in building a space-time diagram 
(and one in particular must host the anchor tile).  Simultaneously, 
the entire $N\times N$ region could itself be considered as a huge tile, 
or \emph{macrotile}.  The \emph{macrocolors} of the macrotile 
are whatever color combinations appear on the boundary of the $N\times N$ 
region.  To control what kind of tileset is realized by the macrotiles, 
tiles use a \emph{wire part} of their colors to transport the bits displayed 
on the outside of the macrotile onto the input tape of the computation zone. 
The algorithm reads the color combination and makes the determination 
whether this kind of macrotile will be allowed (halting if the color combination 
is unsatisfactory).  By design, 
the algorithmic winnowing forces the macrotiles to belong 
to a tileset that is very similar to the original tileset, but with one change:
$N$ is increased so that the algorithm at the next level gets more time to run.
In essence, the algorithm copies its own source code (with the one change in $N$) 
up to the next level.  Then with 
any time left over after checking the color combinations, the algorithm 
can use to do arbitrary other computations (possibly halting for other reasons).

So the tiles organize into macrotiles, the macrotiles organize into 
macromacrotiles, and so on.  When talking about adjacent layers of macrotile, 
we refer to the smaller macrotiles as the \emph{children}.  
An $N\times N$ group of child tiles make up a \emph{parent} tile.  
Two adjacent tiles at the same level are \emph{neighbors}.  If 
two child tiles belong to the same parent tile we call 
the child tiles \emph{siblings}.  The smallest macrotiles (the original 
tileset) are called \emph{pixel} tiles.

The computations which we cause the SFTs to perform in this paper 
are best understood by starting with a simple module, to which 
we add more and more features to obtain more general results. 
We begin here with the most basic module, which simply 
demonstrates the undecidability of the SFT emptiness 
problem using the DRS framework.  

\subsubsection{Basic DRS module}

Given an index $e_0$, we produce a $\mathbb Z^2$-SFT that is 
empty if and only if the $e_0$th Turing machine halts
on empty input.

Fix the zoom sequence $N_i = 2^i$.

Fix a universal Turing machine with binary tape alphabet.  
We can assume that the universal 
machine operates as follows.  It reads the tape up until the 
first 1 and interprets the number of 0's it saw until that point
as the program number.   Then it applies that program to the 
entire original contents of the tape.  In this way the program number itself 
is accessible to the computation.

In the basic module, a \emph{level $i$ macrocolor} is a binary string of length 
$s_i=2\log N_i + 2 + \log k$, where $k$ is the number of colors 
needed to implement a tileset whose configurations
simulate computations from the universal Turing machine,
provided they contain the anchor tile.  Here there will 
be $2\log N_i$ bits for the 
location part, 2 bits for the wire part, and $\log k$ bits 
for the computation part.

Let $A(t)$ be the following algorithm, where $t$ is a binary string input:
\begin{enumerate}
\item (Parsing) 
\begin{enumerate}
\item Start reading $t = 0^e10^i1\dots$, and check that $e<i$.
\item Compute $s_{i}$ and check that the rest of $t$ 
has length $4s_{i}$; interpret the rest of $t$ as colors $c_1,\dots,c_4$.
\end{enumerate}
\item (Consistency) 
\begin{enumerate}
\item Check all four colors have compatible location $(x,y)$
\item Based on the location, check that wires and/or computation 
parts either appear or not, as appropriate.  Set up the $s_{i+1}$-width
wires so that they 
deliver the colors to the computation zone starting at location $(e+i+2)$ of the parent 
tape. 
\item If wire and/or computation bits are used, check they are a valid combination.
\end{enumerate}
\item (Sync Levels) If $(x,y)$ can view the $n$th bit of the parent's tape, check:
\begin{enumerate}
\item If $n<e+i+2$, check the tape has 1 on it if $n=e$ or $n=e+i+2$; otherwise
check the tape has 0 on it.  (This makes the parent computation have $0^e10^{i+1}1\dots$, 
so the parent knows it is at the next level, which is $i+1$)
\item Otherwise, if $n\geq e+ i + 2 + 4s_{i+1}$, check the parent tape is blank.
\end{enumerate}
\item Halt if any checks above fail.
\item Simulate Turing machine $e_0$ on empty input.
\end{enumerate}

All steps except the last one take time at most $\poly(\log N_{i})$, 
equivalently in our case to $\poly(i)$.
Also, the universal machine simulation process only adds 
polynomial overhead.  Therefore, for any $e<i$ and any
 $t$ that begins 
with $0^e10^i1\dots$, the universal machine completes the first 
four steps within $\poly(i)$ time.

Fix $e$ to be the Turing machine index of the algorithm $A$ above.  Let 
$i_0>e$ be large enough that for all $i\geq i_0$ and all inputs $t$
beginning with $0^e10^i1\dots$, steps 1-4 are completed within 
$N_{i-1}/4$ steps on the universal machine.  (Here $N_{i-1}/2$ 
is the size of the computation zone in the macrotile that will be
 running this computation.)

For $i\geq i_0$, let $T_{i}$ be the tileset 
$$T_{i} = \{(c_1,c_2,c_3,c_4) : A(0^e10^i1c_1c_2c_3c_4) \text{ runs at least $N_{i-1}/2$ steps}\}$$

Observe that $i$ is large enough to permit the deterministic 
construction of a 
valid layout of $s_{i+1}$-thickness wires in each $N_i\times N_i$ parent 
tile, and large enough that if any of the checks fail, they fail within 
$N_{i-1}/4$ steps.  Therefore, provided Turing machine $e_0$
runs forever, for each possible location in $N_i\times N_i$, 
there is a $\bar c \in T_i$
which has that location.    
If the location should contain a wire alone then there are exactly 
two $\bar c \in T_i$ 
with that location (one for each bit the wire could carry) 
and if the location should contain 
computation bits then exactly the valid computation fragments
can appear there.
On the other hand, if machine $e_0$ halts, then eventually the tilesets 
$T_i$ become empty.

Recall from the DRS construction the following definition:
a tileset $T$ \emph{simulates} a tileset $S$ at zoom
level $N$ if there is an injective map $\phi:S\rightarrow T^{N^2}$
which takes each tile from $S$ to a valid $N\times N$ array of tiles
from $T$, such that
\begin{itemize}
\item For any $S$-tiling $U$, $\phi(U)$ is a $T$-tiling.
\item Any $T$-tiling $W$ can be uniquely divided into an infinite array of
 $N\times N$ macrotiles from the image of $S$.
\item For any $T$-tiling $W$, $\phi^{-1}(W)$ is an $S$-tiling.
\end{itemize}
Here we have abused notation to let $\phi$ map tilings to tilings in
the obvious way.  

\begin{prop}
For each $i\geq i_0$, $T_i$ simulates $T_{i+1}$ at zoom level $N_i$.
\end{prop}
\begin{proof}
Given $S = \bar d \in T_{i+1}$, map it to the unique $N_i\times N_i$
pattern of $T_i$-tiles whose wires carry data $\bar d$.  The content 
of the wires uniquely determines the content of the computation 
zone because the computation is deterministic, the 
starting tape contents are uniquely determined by $\bar d$,
and $T_i$ only includes tiles which contain valid computation 
fragments.  Since $A(0^e10^{i+1}1\bar d)$ is still running after $N_i/2$-many 
steps, the whole computation zone can be filled with $T_i$-tiles.
\end{proof}

Then by induction, for each $i \geq i_0$, $T_{i_0}$ simulates 
$T_{i}$ at zoom level $\prod_{i_0\leq k <i} N_k$.

\begin{prop}
There is a $T_{i_0}$-tiling if and only if Turing machine $e_0$ 
runs forever.
\end{prop}
\begin{proof}
By compactness, there is $T_{i_0}$-tiling if and only if 
$T_i$ is non-empty for all $i>i_0$.  This occurs
if and only if Turing machine 
$e_0$ runs forever.
\end{proof}

This completes the exposition of the basic module.
For further details we refer the reader to \cite{DurandRomashchenkoShen2012}.

\subsection{Overview of the paper}

The rest of the paper is divided into two parts.  In Secton \ref{sec:rank3}
we construct a $\mathbb Z^2$-SFT of TCPE rank 3, answering 
Barbieri and Garc\'ia-Ramos and laying the foundation for the more 
general results.  

The construction proceeds in stages.  First we define an effectively closed 
subshift $X$ which has TCPE rank 3.  Next we 
describe the computational overlay which allows us to replace 
the infinitely many restrictions defining $X$ with an algorithm 
that will simulate those restrictions.  Finally, we tweak the computation 
framework so that it provides no interference to the local entropy 
properties of $X$.

In slightly more details, the configurations of $X$ consist of \emph{pure 
types}, which are seas of squares, tightly packed together, all the same 
size; and \emph{chimera types}, which contain up to two sizes of squares, 
where the sizes must be adjacent integers.  Infinite, degenerate squares 
also inevitably result; they cannot coexist with finite squares.  Because 
of the chimera types, a configuration of pure type $n$ will form an entropy 
pair with a configuration of pure type $n+1$, but the entropy pairhood 
relation cannot extend to larger gaps.  The infinite, degenerate 
types are connected to the finite types by topological closure only.
However, $X$ is topologically connected enough that the TCPE process 
finishes.  

To show that such a shift is sofic is a straightforward application of the 
DRS framework, simpler than the related square-counting 
construction in \cite{Westrick2017}.  However, no shift with TCPE can contain 
a rigid grid (erasing everything but the grid would yield a non-trivial zero 
entropy factor).  This apparent problem is solved by imagining the entire 
subshift is printed on a piece of fabric, which we then pinch and stretch 
so that the deformed grid itself bears entropy.  The same idea was used 
by Pavlov \cite{Pavlov2018} and we build on his construction.

Finally, a technical problem arises involving the interaction of rare computation 
steps with the need for a fully supported measure.  This problem is solved 
with the notion of a \emph{trap zone}, an idea which originated in 
\cite{DurandRomashchenko-TA}.  The problem is not of any fundamental importance 
and the solution is technical, so it could be skipped on a first reading.

The second part of the paper, Section \ref{sec:rankalpha}, builds heavily 
on the first part and contains all the main results.  Again we build 
an effectively closed subshift, this time 
with a high TCPE rank, superimpose a computation 
to show it is sofic, and put it on fabric to make the computation 
transparent to local entropy.

The TCPE process in $X$ finished quickly because all the pure 
types were transitively chained together.  We can make the process 
finish more slowly by putting topological speed bumps between 
the pure types.  Forbid most of the chimera types of $X$, leaving 
only \emph{regular} chimera types -- those in which the squares occur 
in a regular grid pattern only.  This kind of regular grid pattern is no 
good for connecting pure types, so the entropy pair connections 
are broken.  Since two sizes of square now occur in 
a regular grid, a configuration of chimera type can be parsed 
as a configuration on a macroalphabet where the two macrosymbols 
are the two different squares.  Apply the exact same restrictions 
that defined $X$ to the configurations on this macroalphabet. 
Now the pure types will still get connected, but instead of being 
connected immediately as an entropy pair, they have to wait 
until the TCPE process on the chimera types finishes.  Topological 
speed bumps can be introduced into the chimera types 
of the chimera types, further lengthening the process.

Too many speed bumps, and the TCPE process 
will not finish. But if the speed bumps are organized using 
a well-founded tree, they will not hold things up forever. 
Some ill-founded trees even produce subshifts with TCPE.  We show that
the length of the TCPE process is controlled by the Hausdorff 
rank of the lexicographical ordering on $T \cup [T]$.

Showing that the resulting subshifts are sofic, and superimposing 
the computations transparently to local entropy, requires more 
technical work but contains no surprises.

Finally, all constructions are completely uniform, so in the end 
we can produce a procedure which maps a tree $T$ to 
a $\mathbb Z^2$-SFT $Y$ in such a way that $Y$ 
has TCPE if and only if $T$ is well-founded, and when 
this happens the well-founded rank of $T$ and the 
TCPE rank of $Y$ are related in a predictable way.  This 
simultaneously gives both the $\Pi^1_1$-completeness of 
TCPE and the population of the computable ordinal part of the
hierarchy.

\section{A $\mathbb Z^2$-SFT with TCPE rank 3}\label{sec:rank3}

We begin with constructing a $\mathbb Z^2$-SFT with TCPE rank 3.  
In this construction, many of the features of the general construction 
already appear.   

First we will define an effectively closed $\mathbb Z^2$ subshift
with TCPE rank 3.  Next 
we will argue this subshift is sofic.  Finally, we will show how 
to modify it to obtain a SFT with the same properties.

\subsection{An effectively closed $\mathbb Z^2$ subshift of TCPE rank 3}

\begin{definition}
An $n$-square on alphabet $\{A,B\}$ is an $n\times n$ square of $B$'s, 
surround by an $(n+2)\times(n+2)$ border of $A$'s.
\end{definition}

\begin{definition}\label{def:FAB}
For any alphabet $\{A,B\}$, let $F_{A,B}$ denote a computably enumerable 
set of forbidden patterns which achieves the following restrictions
\begin{enumerate}
\item Every $2\times 2$ block of $B$'s is in the interior of an $n$-square.
\item Every $A$ is part of the border of a unique $n$-square.  
(Infinite, degenerate squares are possible.)  Two such squares 
may be adjacent (see Figure \ref{fig:nestedarrows}), but boundaries may not be shared.
\item If a configuration contains an $n$-square and an $m$-square, then $|m-n|\leq 1$.
\end{enumerate}
\end{definition}

We will show that the subshift $X \subseteq 2^{\mathbb Z^2}$ defined by 
forbidden word set $F_{0,1}$ 
has TCPE rank 3.  In order to do this, we partition $X$ into countably many pieces,
or types, as follows.  The possible types are 
$\omega \cup \{(n,n+1) : n \in \omega\} \cup \{\infty\}$.  To determine the type 
of some configuration $x \in X$, examine what $n$-squares appear in $x$.
\begin{definition}\label{defn:type}  If $x \in X$, the \emph{type} of $x$ is
$$\type(x) = \begin{cases} n & \text{ if $x$ contains only $n$-squares}\\
(n,n+1) &\text{ if $x$ contains $n$-squares and $(n+1)$-squares}\\
\infty &\text{ if $x$ contains no finite $n$-squares for any $n$}\end{cases}$$
\end{definition}

Observe that if $\type(x) = \infty$, then either $x = 1^{\Z^2}$, or the 0's 
which do appear in $x$ appear as the boundaries of up to four 
infinite $n$-squares.  Otherwise, if a finite $n$-square appears, no 
infinite $n$-square may appear, because by part (3), it is forbidden to 
have both an $n$-square and another square that appears to be very large.
And if a square of any other size appears, then again by part (3), the 
sizes can differ by only one, so if any finite square appears, then 
$x$ must have type $n$ or $(n,n+1)$ for some $n$.  The restriction in 
part (1) guarantees that the squares are close together; either touching,
or separated by just a small space.

Now let us identify the entropy pairs.  The fact that squares may be 
either touching or separated by one unit provides freedom for 
gluing blocks.  More precisely, if $i,j \in \Z$, we let $[i,j)$ denote the set 
$\{i, i+1, \dots, j-1\} \in \Z$.  If $v$ is a pattern in $\mathcal L(Y_1)$, 
define the \emph{type} of $v$ as
$$\type v = \{\type(x) : x\in X \text{ and } h(v)\in x\},$$
where $h:Y_1\rightarrow 2^{\mathbb Z^2}$ is the obvious factor map.
We have the following lemma, which shows that any pattern consistent 
with a configuration of a finite type can be extended to a rectangular 
block of completed squares.  The point is that blocks of the kind guaranteed 
below can be placed adjacent to each other freely without breaking any rules 
of $F_{0,1}$.  We state the lemma here but leave the proof for the 
end of the subsection.

\begin{lemma}\label{lem:extension}
For any type $t \in \omega\cup \{(n,n+1): n \in \omega\}$, and any $k\in \omega$,
there is an $N$ large enough so that for all $v \in \Lambda^{[0,k)^2}$, 
if $t \in \type(v)$, 
then for any rectangular region $R$ which contains $[-N, k+N)^2$, we can extend
$v$ to $v' \in \Lambda^{R}$ such that
\begin{enumerate}
\item $t \in \type(v')$
\item $v'$ contains only completed squares 
(every boundary arrow in $v'$ is part of an $n$-square fully contained in $v'$.)
\item The boundary of $v'$ does not contain any two adjacent 1's, nor any 1's 
at a corner.
\end{enumerate}
\end{lemma}

Now we can describe the entropy pairhood facts, which depend only on type.
\begin{lemma}\label{lem:ep-via-type}
If $x,y \in X$, then the following table summarizes exactly when $x$ and $y$ are an entropy-or-equal 
pair (redundant boxes are left blank).
\begin{table}[htbp]
\begin{tabular}{|c|c|c|c|}\hline
\backslashbox{$\type(x)$}{$\type(y)$} & $n$ & $(n,n+1)$ & $\infty$ \\\hline
$m$ & iff $|n-m|\leq 1$ & iff $m \in \{n,n+1\}$ & never\\\hline
$(m,m+1)$ & &iff $m=n$ & never \\\hline
$\infty$ & & & always\\\hline\end{tabular} 
\end{table}
\end{lemma}
\begin{proof}
First we prove all the ``if'' directions.  Suppose $x$ and $y$ have types which 
the table indicates should be entroy-or-equal pair types.  In each of the finite cases, 
there exists a finite type $t$ such that arbitrarily large patterns of $x$ and $y$ 
are each consistent with type $t$.  In the $(\infty, \infty)$ case, for any pair of 
patterns from $x$ and $y$, there is also always a finite type $t$ that is consistent 
with both patterns (any sufficiently large finite type will do).  Given $v$ and $w$ 
$k\times k$ patterns from $x$ and $y$, let $t$ be such a finite type, and let $N$ 
be the number guaranteed by Lemma \ref{lem:extension}.  
Partition the $\mathbb Z^2$ into square plots of side length $2N+k$.
In the center of each plot, place a copy of either $v$ or $w$ (independent choice).
By Lemma \ref{lem:extension}, fill in the rest of each plot in a way consistent 
with $t$.  The result obeys all the rules of 
$F_{0,1}$, so $x$ and $y$ are an entropy pair.

In the other direction, if $x$ and $y$ have types which the table indicates should 
not be entropy-or-equal pairs, that means that $x$ and $y$ have squares of 
size difference more than one (in some cases we are looking at a finite square and 
an infinite, degenerate square, which is also a size difference more than one). 
If $w$ and $v$ are patterns of $x$ and $y$ which are large enough to show 
these too-different squares, then $w$ and $v$ are forbidden from appearing 
in the same configuration, so $x$ and $y$ are not an entropy pair.
\end{proof}

Observe that if $\type(x) = \type(y)$, then $x$ and $y$ are always an entropy pair. 
Therefore, we may represent the entropy pairhood relation $E_1$ by the following 
graph.

\begin{figure}[hbpt]
\includegraphics[width=\textwidth]{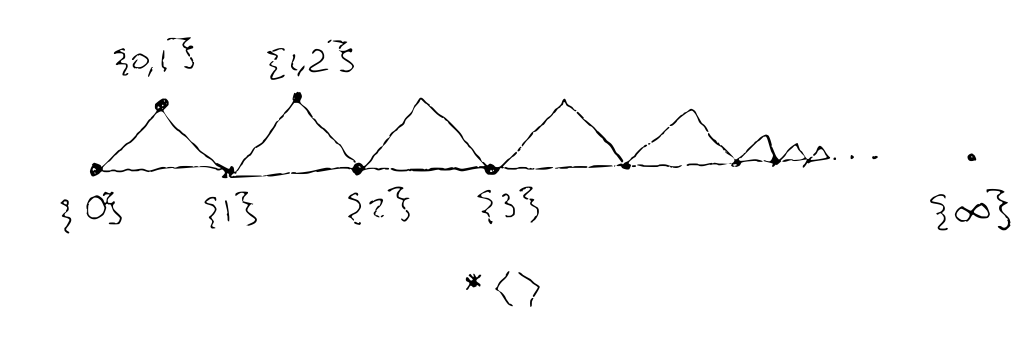}
\end{figure}

Next, $E_2$ is obtained by taking the transitive closure of $E_1$.  
We see that $E_2$ is an equivalence relation with two classes:
$\{x : \type(x) \text{ is finite}\}$ and $\{x : \type(x) = \infty\}$.
Finally, $E_3$ is obtained as the topological closure of $E_2$. 
Observe that $E_3 = X^2$ (every $x$ with infinite type is the 
limit of a sequence of $y$ of increasing finite type).  Therefore, 
the TCPE rank of $X$ is 3.

We conclude this subsection with the proof of Lemma \ref{lem:extension}.
\begin{proof}[Proof of Lemma \ref{lem:extension}]
We first describe a way to construct $v'$ which works if $v$ is nice.  Then we 
explain how to expand any $v$ to a nice $v$.

First, complete any partial squares to produce a pattern consistent with type $t$.  
Call a square $s$ a ``top square'' if there are no squares intersecting the space directly 
above $s$.  Similarly, we have left squares, right squares, and bottom squares. 
The top squares are strictly ordered left to right.  Let us say that the row of top squares 
is \emph{nice} if for every adjacent pair $(s,s')$ of top squares, there is at most 
one pixel gap between the columns that intersect $s$ and the columns that intersect 
$s'$.  We can make a similar definition for the left, right and bottom row of squares.
Call $v$ nice if all four of its rows are nice.

If $v$ is nice, we can extend $v$ to such $v'$ by the following algorithm.  
Directly above each 
top square, place another square of the same size, either with boundaries 
touching, or with a one pixel gap.  Keep adding squares directly above existing ones
until the top of $R$ is reached.  Use the 
freedom of choice in spacing to make sure the topmost square has its top row 
flush with the boundary of $R$.  This is possible if the boundary is at least distance 
$O(n^2)$ away, where $n$ is the size of the top square we started with. 
Do the same for the bottom squares, but going down.  Since $v$ is nice, 
there is no more than one pixel wide gap between these towers of squares, 
so the restrictions $F_{0,1}$ are so far satisfied.

Turn our attention now to the left row.  Let $s$ be the left-most top square.
Then $s$ is also a left square. 
(If it were not a left square, there would be a square that lies 
completely in the half-plane to the left of it; a top-most such square would 
be a top square further left than $s$.) The tower of squares which 
we placed above $s$, together with the tower of squares 
which we placed below the left-most bottom square,
form a natural extension of the left row.  This extended left row is nice, and 
reaches from the top to the bottom of $R$.  Directly to the left of each square in 
the extended left row, place another square of the same size, either with 
boundaries touching, or with a one-pixel gap, strategically chosen.  
Keep adding squares until 
the left boundary of $R$ is reached.  Again we can ensure that the left-most 
added squares have boundary flush with the left edge of $R$.  Doing the 
same on the right side completes the construction in case $v$ is nice.

Now we deal with the case where $v$ is not nice. 
Let $x$ be a configuration of type $t$ in which $v$ appears.  We are 
going to take a larger pattern from $x$ which is nice and 
includes $v$.  To find the top row of this pattern, start with 
a square $s_0$ in $x$ that is located directly above $v$ and is at least 
$k$ distance away from the top of $v$.  Now consider the space directly 
to the right of $s_0$.  There must be a square $s_1$, intersecting this space,
such that the the gap between the right edge of $s_0$ and the left edge of $s_1$ 
is no more than one pixel.  Going in both directions from $s_0$, fix 
a bi-infinite sequence $(s_i)_{i\in\mathbb Z}$ of squares such that for each $i$, 
the square $s_{i+1}$ intersects the space directly to the right of $s_i$,
and there is at most 
one pixel gap 
between the columns that intersect $s_i$ and the columns that intersect 
$s_{i+1}$.  Call this sequence the ``top line'', although it is not a line, 
as it may be rather wiggly.  However, it is roughly horizontal; due to the fact 
that the square sizes are uniformly bounded, 
$s_{i+1}$ intersects the space to the right of $s_i$ by a definite fraction of 
its height.  Therefore, the slope of the line which connects the center of $s_i$ 
and the center of $s_{i+1}$ has magnitude less than $1-\varepsilon$ 
for some $\varepsilon$ depending only on $t$.  
Since $s_0$ is located at least $k$ above $v$, 
an intersection would require a secant slope of magnitude at least 1 
in the top line, so it follows that the top line 
cannot intersect $v$.  Similarly, make a left line, a right line and a bottom line.
The left and right lines always have secant slopes at least $1+\varepsilon$ 
in magnitude.  Due to these approximate slopes, 
the top line and the bottom line must 
each intersect the right line and the left line, and they do so within 
a bounded distance.   Take the 
pattern which consists of the loop made by the four lines 
and all the squares inside that loop.  This pattern is nice.  Apply the argument above.
\end{proof}

\subsection{Enforcing shape restrictions by a SFT on an expanded alphabet}

The shift $X$ from the previous subsection was only effectively closed,
while our goal is to create a SFT with the same properties.  First we show 
how to realize restrictions (1) and (2) of $F_{A,B}$ using 
local restrictions on an expanded alphabet.  Let
$$\Lambda = \{1, \circ, \leftarrow,
 \rightarrow, \uparrow, 
 \downarrow,\ulcorner, \urcorner, \llcorner, \lrcorner, \gleftarrow, \grightarrow,
 \guparrow, \gdownarrow, \gulcorner, \gurcorner, \gllcorner, \glrcorner \}.$$
 We consider the corner pieces to 
also be arrows; corner arrows always point counterclockwise. 
 
 \begin{definition}  Let $Y_1$ be the subshift whose configurations
consist of square blocks of arrows on a background of 1's, 
where
\begin{itemize}
\item Each block of arrows consists of of nested counterclockwise
squares of arrows, the outermost of which are colored gray, while the inner 
are colored white.
\item The $2\times 2$ block of 1's is forbidden.
\end{itemize}
\end{definition}
See Figure \ref{fig:nestedarrows}.   

\begin{figure}[htbp]
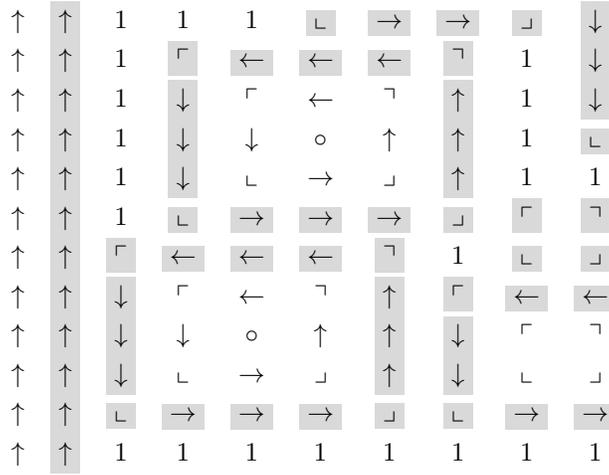

$\begin{array}{cccccccccc}
\uparrow & \guparrow & 1 & 1 & 1 & \gllcorner & \grightarrow & \grightarrow & \glrcorner & \gdownarrow \\
\uparrow & \guparrow & 1 & \gulcorner & \gleftarrow &\gleftarrow & \gleftarrow &\gurcorner & 1 & \gdownarrow\\
\uparrow & \guparrow & 1 & \gdownarrow & \ulcorner & \leftarrow & \urcorner & \guparrow & 1 & \gdownarrow\\
\uparrow & \guparrow & 1 & \gdownarrow & \downarrow & \circ & \uparrow & \guparrow & 1 & \gllcorner\\
\uparrow & \guparrow & 1 & \gdownarrow & \llcorner & \rightarrow & \lrcorner & \guparrow & 1 & 1\\
\uparrow & \guparrow & 1 & \gllcorner & \grightarrow &\grightarrow & \grightarrow &\glrcorner & \gulcorner & \gurcorner\\
\uparrow & \guparrow & \gulcorner & \gleftarrow & \gleftarrow & \gleftarrow & \gurcorner & 1 & \gllcorner & \glrcorner \\
\uparrow & \guparrow & \gdownarrow & \ulcorner & \leftarrow & \urcorner & \guparrow &\gulcorner &\gleftarrow &\gleftarrow \\
\uparrow & \guparrow & \gdownarrow & \downarrow & \circ & \uparrow & \guparrow &\gdownarrow &\ulcorner &\urcorner \\
\uparrow & \guparrow & \gdownarrow & \llcorner & \rightarrow & \lrcorner & \guparrow &\gdownarrow &\llcorner &\lrcorner \\
\uparrow & \guparrow & \gllcorner & \grightarrow & \grightarrow & \grightarrow & \glrcorner &\gllcorner &\grightarrow &\grightarrow \\
\uparrow & \guparrow & 1 & 1 & 1 & 1 & 1 & 1 & 1 & 1 \\
\end{array}$
\caption{A permitted pattern from an element of $Y_1$}\label{fig:nestedarrows}
\end{figure}

\begin{prop}
The set of $2 \times 2$ patterns from $\Lambda$ which never appear 
in $Y_1$, when taken as a set of forbidden patterns, define $Y_1$ as a SFT.
\end{prop}
\begin{proof}
This is essentially identical to the proof of Lemma 1 in \cite{Westrick2017}.  
Note that the $2 \times 2$ restrictions suffice to guarantee that each arrow is 
continued by another arrow of the same color, forming oriented 
paths which make only left turns.  The $2\times 2$ restrictions also suffice 
to guarantee that on the inner edge of each oriented path, there is 
a tightly nested white path of the same 
orientation; that on the outer edge of each oriented white path there is 
a tightly nested path of the same orientation; and that on the outer edge 
of each oriented gray path we find only 1's and oppositely-oriented
gray paths.  Starting from a finite gray path and following its nested paths 
inward, we must terminate at nothing, a $\circ$, or a $\begin{array}{cc} \ulcorner & \urcorner \\ \llcorner & \lrcorner\end{array}$, which guarantees that the gray path was 
a square.  Infinite gray paths must be either one-sided or two-sided; the 
three-sided case is ruled out by following nested white paths inward 
until one of the forbidden configurations $\begin{array} {cc} \downarrow & \uparrow\end{array}$ or 
$\begin{array}{c} \leftarrow \\ \rightarrow\end{array}$ is found.
\end{proof}

 Observe that 
 if we let $h$ be the factor map which takes white symbols to 1 and gray symbols
 to 0, then $h(Y_1)$ is precisely the shift given by the restrictions (1) and (2) of $F_{0,1}$
from Definition \ref{def:FAB}.
In the next subsection we realize restriction (3) with SFT computation.

\subsection{Enforcing size restrictions with SFT computation}\label{sec:drs-basic}

We now show that the subshift $X$ defined by restrictions
$F_{0,1}$ is sofic, using 
a simple application of the Durand-Romashchenko-Shen 
SFT computation framework. 
This framework is likely overpowered for this application,
because there is no arbitrary algorithm appearing 
in the definition of $X$.
However, in the general case we must have complete 
computational freedom, 
so we introduce it now in the simpler setting.  We begin 
with the SFT $Y_1$ on alphabet $\Lambda$ 
 defined above, with all $2\times 2$ restrictions.
 We will use a superimposed computation 
 to realize the square size 
 restrictions.

We now describe a modification of the basic DRS 
module.  As in the introduction,
we will ultimately choose some large $i_0$,
form the tileset $T_{i_0}$,
and superimpose a tile from $T_{i_0}$ onto 
each symbol of $\Lambda$ (subject to some restrictions).
Note that the side-length of a macrotile at level $i_0$ is 
one pixel, while the side-length of a macrotile at level 
$i>i_0$ is $L_i := \prod_{i_0 \leq j < i} N_j$ pixels.

\begin{definition}We define the \emph{responsibility zone} 
of a level $i$ macrotile $M$ as follows
\begin{enumerate}
\item If $i=i_0$, the responsibility zone is the pixel tile $M$ itself.
\item If $i>i_0$, consider the $(N_{i-1}+2)$-side-length
square of level $i-1$ macrotiles which concentrically contains $M$. 
The responsibility zone of $M$ is the union of the responsibility 
zones of these level $i-1$ macrotiles.
\end{enumerate}
\end{definition}

Equivalently, the responsibility zone of a level $i$ macrotile $M$
is the square of pixel-side-length 
\mbox{$R_i :=(L_i + \sum_{j=i_0}^{i-1} 2L_j)$}
which concentrically contains $M$.

Macrotiles will need to ``know'' (have encoded 
somehow in their colors) what is going on inside 
their responsibility zone.  Recall that the macrotiles will 
be superimposed on $Y_1$-configurations.  Therefore, 
it makes sense to discuss $Y_1$-squares and their 
relative location to the macrotiles.

\begin{definition}
A macrotile $M$ is \emph{responsible for} a $Y_1$-square 
if there are macrotiles $M_1$ and $M_2$, both at the same level,
such that
\begin{enumerate}
\item $M_1$ and $M_2$ are fully contained in the responsibility 
zone of $M$,
\item $M_1$ and $M_2$ are either same tile or neighbors, and
\item the square has one corner in $M_1$ and another corner in $M_2$.
\end{enumerate}
\end{definition}

\begin{prop}\label{prop:responsibility}
If $M$ is a macrotile of level $i$, and $S$ is a $Y_1$-square, 
then $M$ is responsible for $S$ if and only if one of the following 
holds.
\begin{enumerate}
\item There is a level $i-1$ macrotile $M_1$ within the responsibility 
zone of $M$, and $M_1$ is responsible for $S$.
\item There are two adjacent level $i-1$ macrotiles $M_1,M_2$ within
the responsibility zone of $M$, and $S$ has one corner in each  of $M_1$ 
and $M_2$.
\item There are two (possibly non-adjacent) level $i-1$ macrotiles 
$M_1,M_2$ which are both children of $M$, and $S$ has 
one corner in each of $M_1$ and $M_2$.
\end{enumerate}
\end{prop}
\begin{proof}
It is clear that in each of the three cases above, $M$ is responsible for $S$. 
Suppose $M$ is responsible for $S$ via level $j$ macrotiles $M_1$ and $M_2$,
where $j$ is chosen as small as possible.  If $j=i$, then $M_1=M_2=M$ and 
since no smaller $j$ works, it must be that $S$ has one corner in each of 
two non-adjacent children of $M$, so case (3) above applies.  If $j=i-1$ 
and $M_1=M_2$, then $M_1$ is responsible for $S$, so case (1) above applies.
If $j=i-1$ and $M_1\neq M_2$, then case (2) above applies.  Finally, if $j<i-1$ 
then $M_1$ and $M_2$ both lie in the responsibility zone of a single level 
$i-1$ macrotile, so case (1) above applies.
\end{proof}

We need a macrotile $M$ to ``know about'' all the sizes of square that it
is responsible for.  We also need $M$ to know about the location of all partial corners or 
partial sides of squares that intersect $M$, but are too big for $M$ 
to be responsible for them.  Each macrotile will tell its neighbors 
what it knows.  Some of what it knows will be told to all neighbors,
other things it knows will be shared only with its siblings.  

In the basic construction, all the information displayed in 
the colors could be thought of as information that is 
shared equally between the two macrotiles whose 
adjacent sides have that color.  Going forward, it becomes 
convenient to think of some parts of the macrocolor as having a
direction of information flow.  For example, the first
bits of a color could be for information that is flowing 
rightward (for a vertical color) or upward (for a horizontal color), 
while the later bits could be for information that 
is flowing leftward or downward.  In this way, a macrotile can tell its 
neighbors what it knows by using the first part of the color 
on its right and top edges, and the second part of the color 
on its left and bottom edges.  It uses the second part of 
its right and top colors, and the first part of its left and bottom 
colors, to receive messages sent by its neighbors.

With this information flow idea in mind, 
a level $i$ macrocolor for a tile $M$ is a binary 
string which contains the following information:
\begin{enumerate}
\item Location, wire and computation bits as in the basic module (synchronized, space $O(\log N_i)$)).
\item Self-knowledge bits (sending):
\begin{enumerate}
\item A \emph{type}, which is a list of up to two numbers less than or equal to $L_i$, representing all sizes for squares that $M$ is responsible for.  (space: $O(\log L_i)$) 
\item If $i>i_0$, a \emph{justification} for each number above.  The justification is the location of 
the lexicographically least 
level $i-1$ macrotile from $M$'s responsibility zone which reported this information to $M$. The method of reporting is described later. (space: $O(\log N_i)$)
\item A list of up to 4 pixel-locations, measured relative to $M$, of corners of large squares,
together with orientation information for the corner (space: $O(\log L_i)$)
\item A list of up to 2 pixel-locations, measured relative to $M$, of sides of large squares that pass through $M$, together with orientation information.  Since the side passes 
all the way through, the location is just a single $x$-pixel-coordinate or $y$-pixel-coordinate. (space: $O(\log L_i)$)
\end{enumerate}
\item Neighbor-knowledge bits (receiving): The tile $M$ receives self-knowledge 
information as above from its neighbor.
\item Diagonal-neighbor-knowledge bits (sending): The tile $M$ will tell its neighbor $M'$
about the self-information of the two tiles which are adjacent to $M$ and diagonal to $M'$.
\item Diagonal-neighbor-knowledge bits (receiving): The tile $M$ will receive from 
its neighbor $M'$ the self-information of the two tiles which are adjacent to $M'$ 
and diagonal to $M$.
\item Parent-knowledge bits (synchronized, shared only with siblings).
\item Corner message passing bits (outgoing, shared only with siblings).  If $M$'s
self-information includes a partial corner, it sends out the pixel location,
relative to its parent, of that corner 
in the directions of the two arms of the corner.  Or, if $M$ receives a corner message
on one side, it sends the same message out on the opposite side.
\item Corner message receiving bits (incoming, shared only with siblings).
\end{enumerate}

The combined self-knowledge bits make it so that the computation going on inside $M$
will be aware of what is happening in the responsibility zone of $M$'s 
eight neighbors.

The parent-knowledge bits allow $M$ and all of its siblings to be 
aware of what is happening in the responsibility zone of their shared parent.

The corner message passing bits
allow $M$ to use its colors to communicate long distances with 
its siblings, so they can together ascertain the sizes of any
square which has two corners in their shared parent.
This is done by the same method as in \cite{Westrick2017}.  
Any macrotile with a partial corner
must send out a message 
containing the deep coordinate of that side.  To send a message in 
one particular direction (north, south, east or west), 
the macrotile displays the message in 
its macrocolor on just one side.  
Any sibling macrotile receiving the message must
pass the message on (unless the parent boundary is reached, 
in which case the message stops); 
eventually the message may reach 
a sibling macrotile with a matching corner.  Since the recipient 
macrotile also knows the deep coordinates of its own corner, 
it can calculate the distance the message traveled, which 
is the side length of the square.

We now informally describe the rest of
what the algorithm running 
inside $M$ does with all the color information above, followed
by a step-by-step summary of the algorithm.

Based on the color information above, $M$ has many 
ways to learn about sizes of squares that could be in the responsibility 
zone of $M$'s parent.  Macrotile $M$ makes sure 
that its parent-knowledge includes:
\begin{enumerate}
\item Any size that appears in $M$'s self-knowledge or any of 
$M$'s neighbors self-knowledge (this takes care of all squares 
that $M$ and each of its neighbors are responsible for).
\item The combined self-knowledge of $M$ and its eight 
neighbors is also enough to deduce whether there are 
neighbor squares $M_1,M_2$ among these nine, and 
a $Y_1$-square $S$, such that $S$ has one corner in $M_1$
and another corner in $M_2$.  If this occurs, $M$ must 
also make sure that its parent has recorded the size of $S$.
\item Based on the corner message passing, $M$ may learn 
about a square which has one corner in $M$ and another 
corner in a sibling of $M$.  If this happens $M$ must make sure 
its parents knows the size of that square.
\item On the other hand, if $M$ sends corner messages out 
and gets none back, it should make sure that its parent-knowledge
includes knowledge of that corner.
\item If the parent cites $M$ as a source of any of its knowledge, 
$M$ must check that this is true.  If the parent cites 
another source for knowledge provided by $M$, 
then $M$ must check this other source has a location 
lexicographically less than $M$'s own location.
\end{enumerate}

Proposition \ref{prop:responsibility} implies that 
steps (1)-(3) above are enough to ensure that
if $M$'s parent is responsible for a square $S$,
and if all macrotiles at the same level as $M$ 
have accurate self-information, then
$M$'s parent will have accurate self-information.

Now we summarize the algorithm run by macrotiles at all levels. 
Let $A(t)$ be the algorithm which does the following:
\begin{enumerate}
\item (Parsing)
\begin{enumerate}
\item Start reading $t = 0^e10^{i_0}10^i1\dots$ and check that 
$e< i_0\leq i$.
\item Check the rest of $t$ has the right length for four level-$i$ 
macrocolors.
\end{enumerate}
\item (Consistency)
\begin{enumerate}
\item Check location, wire, and computation parts are consistent 
(same as basic module).
\item Check that the same self-knowledge appears in all four colors, 
that all outgoing diagonal-neighbor-knowledge agrees with 
corresponding incoming neighbor-knowledge, and that the
incoming diagonal-neighbor-knowledge is consistent.
\item Check the parent-knowledge is the same on all colors shared 
with siblings, and blank on all colors shared with non-sibling 
neighbors.
\item If self-knowledge indicates a corner, use 
$i_0$, $i$, and the location bits to compute the pixel location 
of the corner relative to the parent.  Check that messages 
are sent out along the corner arms; the content of the message
is the pixel location of the corner relative to the parent.  
If corner messages are incoming,
make sure they either match a corner from the self-knowledge,
or are sent outgoing on the opposite side (exception: do not 
send corner messages to non-siblings).
\item If any self-knowledge, neighbor-knowledge, or 
diagonal-neighbor-knowledge includes $n$ in its type, make sure 
the parent-knowledge includes $n$ in its type.
\item If the combined self-, neighbor-, and diagonal-neighbor-knowledge
reveals that among these 9 macrotiles there is a square and a pair of neighbors 
which each contain a corner of that square, compute the 
size of the square, and make sure the parent-knowledge 
includes that size in its type.
\item If the self-knowledge has a corner and receives a matching
corner-message, compute the size of the associated square and 
make sure the parent-knowledge includes that size in its type,
and that the parent-knowledge does not include that corner.
If the self-knowledge has a corner and does not receive a matching
message, make sure the parent-knowledge includes that corner.
If the self-knowledge has a partial side, make sure the parent-knowledge 
has that partial side if and only if no messages came along it 
in either direction.
\item If the parent cites the location of this macrotile as justification 
for a size, make sure that is true.
\item If the parent cites the location of another macrotile for information 
that this macrotile provided, check that the cited location is lexicographically
less than the location of this macrotile.
\item If the parent claims a partial corner or side that would intersect 
this macrotile, make sure that is true.
\item Check that the type of the parent-knowledge is either empty, 
contains a single number, or contains two adjacent numbers $n, n+1$.
If there is a number $n$ and the parent-knowledge also 
contains a partial corner or side, halt if the partial corner or side is 
already too long compared to $n$.
\end{enumerate}
\item (Sync Levels)
\begin{enumerate}
\item If the location of this tile is on the parent tape, check the 
parent tape begins with $0^e10^{i_0}10^{i+1}1\dots$ and
ends with blanks.
\item If the location of this tile is on the parent tape and can view
the parent's self-knowledge part on the parent tape, 
check that what is seen agrees with 
the parent-knowledge recorded in the colors $c_i$.
\end{enumerate} 
\item If any of the above steps do not check out, halt.  Otherwise, run forever.
\end{enumerate}

When $e\leq i_0\leq i$, the steps described take $\poly (\log L_{i+1})$ time to run,
equivalently in our case to $\poly(i)$.
Fix $e$ to be the index of the algorithm described above; let $i_0\geq e$ 
be large enough that for all $i\geq i_0$, if $t$ begins with 
$0^e10^{i_0}10^i1\dots$, then $A(t)$ finishes all the checks within
$N_{i-1}/2$ steps.  For $i\geq i_0$, define
$$\hat T_i = \{(c_1,c_2,c_3,c_4) : A(0^e10^{i_0}10^i1c_1c_2c_3c_4) \text{ runs forever}\}$$
Let $T$ denote the set of all valid $T_{i_0}$-tilings.

Next we are
going to define a SFT subset of $T \times Y_1$ 
which has $X_{0,1}$ as a factor.  To specify 
which subset, we define a notion of accuracy that can hold between the 
$T$-part and the $Y_1$-part of our SFT.  Recall that 
$R_i$ denotes the pixel-side-length of the responsibility zone 
for a macrotile of level $i$.  

\begin{definition}
Given the self-knowledge $s$ of an $i$-level macrocolor,
 and $\bar a$ a permitted pattern from $Y_1$ of size $R_i \times R_i$,
we say that $s$ is \emph{accurate for} $\bar a$ if, 
considering a level-$i$ macrotile $M$ superimposed 
centrally in $\bar a$,
\begin{enumerate}
\item The type of $s$ equals the set of sizes of squares for which 
$M$ is responsible,
\item Justification of $s$ are accurate: if $n$ is justified by giving the 
location of a level $i-1$ macrotile $M'$, then $M'$ must be the 
lexicographically least child of $M$ which either is 
\begin{itemize}
\item itself 
responsible for an $n$-square, or 
\item has one corner of an $n$-square, the other corner of which 
is contained in another child of $M$, or
\item there is an $n$-square with one corner in each of two adjacent
level $(i-1)$ macrotiles from the $3\times 3$ block of macrotiles 
centered at $M'$.
\end{itemize}
\item Corner and partial side locations of $s$ agree exactly with 
the actual locations of partial corners and sides in the central $L_i\times L_i$ 
region of $\bar a$.
\end{enumerate}
\end{definition}

We will also say that a the self-information of a macrotile 
$\bar c \in \hat T_i$ is accurate for $\bar a$ in cases where 
$\bar a$ is larger than the responsibility zone of a level-$i$ macrotile, 
in cases where it is clear precisely 
where we intend to superimpose $\bar c$ on $\bar a$.

\begin{definition}\label{def:accuracy}
Let $y \in Y_1$, let $u$ be a $L_i\times L_i$ subset of $\mathbb Z^2$,
and let $\bar c$ be a level-$i$ macrocolor. We say that 
\begin{enumerate}
\item $\bar c$ is \emph{self-accurate} at $u$ in $y$ if the self-information of $\bar c$ is 
accurate for $y$ when $\bar c$ is superimposed on $u$.
\item $\bar c$ is \emph{accurate} at $u$ in $y$ if $\bar c$ is self-accurate
and additionally all 8 
neighbor-informations of $\bar c$ are accurate for $y$ when superimposed on the 
corresponding eight $L_i\times L_i$ regions concentrically surrounding $u$.
\end{enumerate}
\end{definition}

\begin{definition}\label{def:tile-accuracy}
Let $y \in Y_1$, let $u$ be a $L_i \times L_i$ subset of $\mathbb Z^2$,
and let $\bar t$ a permitted pattern from $T$ consisting of a single a level-$i$ 
macrotile.  We say that
\begin{enumerate}
\item $\bar t$ is \emph{self-accurate} at $u$ in $y$ if the top-level macrocolor 
of $\bar t$ is self-accurate at $u$ in $y$, and for each $j<i$ and each 
level-$j$ macrotile appearing in $\bar t$, that macrotile is accurate at its 
location in $y$.
\item $\bar t$ is \emph{accurate} at $u$ in $y$ if it is self-accurate, and 
additionally its top-level macrocolor is accurate.
\end{enumerate}
\end{definition}

\begin{prop}\label{prop:up-accurate}
Given $y \in Y_1$, 
and $x \in T$,
if each $T_{i_0}$ tile of $x$ is self-accurate at its location in $y$, then
whenever $x \uhr u$ is a macrotile, it is accurate at $u$ in $y$.
\end{prop}
\begin{proof}
By induction.  Assume that $i_0\leq j < i$ and all level-$j$ macrotiles 
in $x$ have self-information that is 
accurate at their location in $y$.  Then each level-$j$ tile is accurate,
because its neighbor-information is copied from its self-accurate neighbor
tiles.
Let $M$ be a level-$(j+1)$ macrotile.  The children of $M$ which live on the 
tape of $M$ have made sure $M$'s self-knowledge is in the parent 
information part of $M$'s children.
  
Any child of $M$ whose self-information includes $n$ in its type, has forced $M$ 
to include $n$ in its type.  Any child of $M$ who has a square of 
length $n$ straddling two level-$j$ tiles in its $3 \times 3$ neighborhood,
 has discovered this and forced $M$ to include $n$ in its type. 
Any child of $M$ who has a corner of a square of length $n$, the other 
corner of which is in another sibling, has found it via corner message 
passing, and forced $M$ to include $n$ in its type. 
Any child of $M$ who has a corner or partial side which did not collect 
matching messages, has forced $M$ to include that corner or 
partial side.

Conversely, $M$ must include justifications for its type, and each child does
check to make sure that if $M$ uses that child as justification, that this is
accurate.  Children located on partial corners/sides claimed by $M$ 
also check to make sure $M$ has that information correct.
\end{proof}

The preceding proposition shows that the requirement to self-accurately superimpose
$\hat T_{i_0}$-tiles onto element of $\Lambda$ suffices to ensure 
that if any tiling is created, then the macrotiles at all levels have 
accurate information about the squares upon which they are superimposed.
In other words, accuracy at the bottom level flows up.  This motivates 
the following definition.

\begin{definition}
Let $\Lambda_2 = \{(\bar c, a) \in \hat T_{i_0}\times \Lambda : \bar c \text{ is self-accurate for } a\}$.
Let $Y_2$ be the SFT whose restrictions are the color-matching 
restrictions from $\hat T_{i_0}$, and the $2\times 2$ restrictions 
from $Y_1$ on the alphabet $\Lambda$.  
\end{definition}

Towards showing that $X_{0,1}$ is indeed a factor of $Y_2$ via the natural factor 
map, we need the following lemmas.

\begin{lemma}\label{lem:adjacent}
Given $y \in Y_1$, suppose that $\bar t_1$ and $\bar t_2$ are 
level-$i$ macrotiles from $T$ which are self-accurate 
at locations $u_1$ and $u_2$ in $y$.
Then if $u_1$ and $u_2$ are adjacent and the top-level colors 
of $\bar t_1$ and $\bar t_2$ are permitted adjacent, then $\bar t_1$ 
and $\bar t_2$ are permitted adjacent in $T$.
\end{lemma}
\begin{proof}
By induction.  The $i=i_0$ case is trivial.  Suppose the lemma 
holds for each macrotile of level $i-1$.  Without loss of generality suppose 
that $\bar t_1$ is left of $\bar t_2$.  It suffices to show that 
$\bar s_1$ is permitted adjacent to $\bar s_2$ whenever $s_1$ 
is a child macrotile of $\bar t_1$ on its right edge and
$\bar s_2$ is the corresponding child macrotile of $\bar t_2$ on its 
left edge.  By definition of self-accuracy, all child macrotiles
found in $\bar t_1$ and $\bar t_2$ are accurate at their location 
in $y$.  So by induction it suffices to check that $\bar s_1$ and 
$\bar s_2$ have matching top-level colors.  Their locations 
match by the way they were chosen.  Their wire bits match, if 
they have wire bits, because the top-level colors of $\bar t_1$
and $\bar t_2$ match.  They do not display computation bits. 
Their self-knowledge and neighbor-knowledge bits match because 
all these fields are set to their unique 
accurate values.  Finally, no parent-information
nor corner-message passing bits are displayed across the parent 
boundary, so they match by both displaying nothing.
\end{proof}

\begin{lemma}\label{lem:down-accurate}
Suppose we have an element $y\in Y_1$ whose squares 
have not-too-different sizes, and an $L_i\times L_i$ region
$u \subseteq \mathbb Z^2$.  Then for any level-$i$ macrocolor 
$\bar c$ that is self-accurate at location $u$ in $y$, there 
is a level-$i$ macrotile $\bar t$ from $T$ which is self-accurate 
at $u$ in $y$ and has top-level macrocolors $\bar c$.
\end{lemma}
\begin{proof}
The proof is by induction on $i$,
and the $i_0$ case is trivial.  Given the level $i$-macrocolor
$\bar c$,
consider $u$ as a union of $L_{i-1}\times L_{i-1}$ regions
corresponding to child macrotiles of level $i-1$.
For each such child
macrotile $M$, we assign an accurate macrocolor as follows.
\begin{enumerate}
\item Set the all the self-knowledge, 
neighbor-knowledge and diagonal-neighbor-knowledge
of $M$
to be accurate in $y$.  All these knowledge fields 
are completely deterministic.  
\item Each $M$ copies its parent-knowledge field from $\bar c$.
\item Set the corner-messages also to reflect the reality of $y$:
if $M$ contains a partial corner or side which has a matching 
corner in the same parent, $M$ should receive a message; 
if $M$ receives a message or has a corner, $M$ should send 
a message.
\item The location part of $M$'s colors is uniquely determined by 
its location in $u$.
\item The wire parts and computation parts of 
$M$'s colors are uniquely determined by 
$\bar c$.
\end{enumerate}
Observe these macrocolors are designed so that they are accurate 
and so that any two child macrotiles with adjacent locations 
have matching macrocolors on the side they share.  By induction 
assume we have a self-accurate macrotile with the chosen macrocolor at each 
child location.  By Lemma \ref{lem:adjacent}, these child macrotiles 
are permitted adjacent, so together they form a level-$i$ macrotile 
$\bar t$ from $T$.  The wire parts of the child macrocolors were 
chosen to ensure that $\bar t$'s top macrocolor is $\bar c$.
\end{proof}

\begin{definition}
Let $h_2:Y_2\rightarrow 2^{\mathbb Z^2}$ be the factor map 
which maps $(\bar c, a)$ to 0 if $a$ is a boundary symbol and 1 otherwise.
\end{definition}

That is, $h_2$ recovers the square boundaries from the $Y_1$ 
part of $x \in X_2$.

\begin{prop}\label{prop:Xfactor} The subshift $X_{0,1}$ is sofic, 
and furthermore $h_2(Y_2) = X_{0,1}$.  
\end{prop}
\begin{proof}  
Given $y \in Y_1$ with squares of not-too-different sizes,
by Lemma \ref{lem:down-accurate}, arbitrarily large accurate 
macrotiles from $T$ 
can superimposed on a central region of $y$.  
It follows by compactness that there is an infinite $T_{i_0}$-tiling
which can be superimposed accurately on $y$.  In particular, 
the superposition is self-accurate at the pixel level.  Therefore there is 
$x \in Y_2$ such that $h_2(x)$ gives the square outlines of $y$.

On the other hand, consider an element $y \in Y_1$ with squares of size 
difference more than 1 (the case of a finite square and an infinite square 
is included in this possibility).  Suppose for contradiction that there 
is $x \in T$ such that each pixel of $x$ is self-accurate at its location 
in $y$.  Then there is a 
macrotile $M$ in $x$ large enough to be responsible for two squares 
of too-different sizes (or large enough to be responsible for a finite 
square and a very long partial corner or side of an infinite square).
By Proposition \ref{prop:up-accurate}, the part of $x$ which corresponds
to the responsibility zone of $M$ is accurate.  Therefore, the self-knowledge
of $M$ includes the two squares.  Therefore the children of $M$
see the two squares in their parent-knowledge.  Therefore the children 
of $M$ halt at the last consistency step of the algorithm.  Contradiction.
\end{proof}

If there is a macrotile that 
has $n$ written on its parameter tape, 
then for every pixel location $t \in \mathbb Z^2$, every sufficiently 
large macrotile containing $t$ also has $n$ written on its
 parameter tape.  Therefore, an element of $Y_2$ could 
 be said to have a limiting \emph{information type}, which is 
 by definition equal to 
 the collection of all sizes recorded on any parameter tape, 
 or $\infty$ if no finite size is ever recorded.   Observe that the 
 information type of an element $x\in Y_2$ is equal to
 $\type(h_2(x))$.

\subsection{Preserving TCPE rank}
We have seen that $X_{0,1}$ is sofic, but so far only via an SFT extension 
$Y_2$ that does not have TCPE.  The DRS-type 
construction is tiling-based, so any subshift that uses it 
factors onto a zero-entropy subshift that retains only the 
tiling structure.  Below, we describe how to modify $Y_2$ 
to fix this problem.

The idea is to imagine the tiling structure is printed on 
a piece of fabric, which can be pinched and stretched 
so that the alignment of tiles in one region has no bearing 
on the alignment of tiles far away.  Now the tiling structure 
itself bears entropy (information about pinching and stretching), 
so any factor map which retains 
any part of the computation also retains the 
entropy of the tiling structure on which the computation lives. 
This idea is made precise below with an analysis
of a construction by Pavlov \cite{Pavlov2018}.  He shows the 
following.

\begin{theorem}[Pavlov \mbox{\cite[Theorem 3.5]{Pavlov2018}}]
For any alphabet $A$, there is an alphabet $B$ and a map taking 
any orbit of a point in $A^{\Z^2}$ to a union of orbits of points in 
$B^{\Z^2}$ with the following properties:
\begin{enumerate}
\item $O(x) \neq O(x') \implies f(O(x)) \cap f(O(x')) = \emptyset$.
\item If $W \subseteq A^{\Z^2}$ is a subshift (resp. SFT), then $f(W)$
is a subshift (resp. SFT).
\item If $W$ is a $\Z^2$-subshift with a fully supported measure, 
and there exists an $N$ such that for every $w, w' \in L(W)$, 
there are patterns $w=w_1,w_2,\dots, w_N=w$ such that for all 
$i \in [1, N)$, $w_i$ and $w_{i+1}$ coexist in some point of $W$, 
then $f(X)$ has TCPE.
\end{enumerate}
\end{theorem}

However, in light of the subsequent work by Barbieri and 
Garc\'ia-Ramos \cite{BGR-TA}, stronger claims 
should be made about Pavlov's construction. 
We make the following definition.

\begin{definition} Let $W$ be a $\Z^2$-subshift.  We say $(x,x')\in W$ is 
a \emph{transitivity pair} if for every pair of patterns $v, v'$ that 
appear in $x$ and $x'$ respectively, $v$ and $v'$ coexist
in some point of $W$.
\end{definition}

Examination of 
Pavlov's proof shows he also proved the following.  The point is 
that this transformation preserves many things about 
an SFT with a fully supported measure, while upgrading all 
transitivity pairs into entropy-or-equal pairs.

\begin{theorem}[essentially Pavlov \mbox{\cite[Theorem 3.5]{Pavlov2018}}]\label{thm:pavlov-modified}
For any alphabet $A$, there is an alphabet $B$ and a map taking 
any orbit of a point in $A^{\Z^2}$ to a union of orbits of points in 
$B^{\Z^2}$ with the following properties:
\begin{enumerate}
\item $O(x) \neq O(x') \implies f(O(x)) \cap f(O(x')) = \emptyset$.
\item If $W$ is a subshift (resp. SFT), then $f(W)$
is a subshift (resp. SFT).
\item If $W$ is a subshift with a fully supported measure, 
then $f(W)$ has a fully supported measure.
\item If $W$ is a subshift with a fully supported measure, and $x,x' \in W$, 
the following are equivalent:
\begin{enumerate}
\item[(i)] $(x,x')$ is a transitivity pair in $W$.
\item[(ii)] $(y,y')$ is an entropy-or-equal 
pair  in $f(W)$ for some $(y,y')\! \in\! f(x)\!\times\!f(x')$
\item[(iii)] $(y,y')$ is 
an entropy-or-equal pair in $f(W)$ for every $(y,y')\! \in\! f(x)\!\times\! f(x')$
\end{enumerate}
\item If $W$ is a subshift, 
$A \subseteq W$ is a shift-invariant set, and $x \in W$, then 
$x\in \overline A$ 
if and only if $f(x)\subseteq \overline{f(A)}$.
\end{enumerate}
\end{theorem}
\begin{proof}
The proofs of (1)-(3) are exactly as in the original.  The proof of (4) is 
also essentially there, if a bit roundabout.  Here we sketch a direct route to 
(4), using the same language as in the original proof.

$(i)\Rightarrow (iii)$.  
First, suppose $(x,x')$ is a transitivity pair.  Let $(y,y') \in f(x)\times f(x')$,
with $y \neq y'$, 
and fix $S$ and $w=y \uhr S$ and $w'=y'\uhr S$ such that $w \neq w'$.  
Let $v,v'$ be patterns in $W$ such that $v$ and $v'$
induce $w$ and $w'$.  Then $v$ appears in $x$ 
and $v'$ appears in $x'$, so there is a point $x'' \in W$ in which 
$v$ and $v'$ both appear.  Because $W$ has a fully supported 
measure, we may assume that a pattern $v'''$ containing both 
$v$ and $v'$ appears with positive 
density in $x''$. 
Let $y'' \in f(x'')$ be chosen so that all ribbons of $y$ are 
perfectly horizontal or vertical and spaced 3 apart.  
Pick an infinite, positive density subset of 
patterns in $y''$ which are induced by $v'''$ and which 
are far enough apart.  At each of these locations, 
finitely perturb the ribbons to wiggle a copy of 
$w$ or $w'$ (independent choice)
into those locations.  The independent choice is possible 
because the locations are far enough apart.

$(ii) \Rightarrow (i)$.  If $(y,y') \in f(x) \times f(x')$ is 
an entropy-or-equal pair, it is a transitivity pair.  
So if $v$ and $v'$ are patterns that appear in $x$ 
and $x'$,  they induce $w$ and $w'$ in $y$ and $y'$,
and there 
is some $y'' \in f(W)$ in which $w$ and $w'$ coexist.  
Then $v$ and $v'$ 
coexist in each element of $f^{-1}(y'')$.

For (5), the forward direction follows for any given 
$y \in f(x)$ by considering elements of $f(A)$ 
that have the same ribbon structure\footnote{In \cite{Pavlov2018},
``the first two coordinates of $y$''.} as $y$,
and the reverse direction also follows by
restricting attention to a single fixed ribbon 
structure.
\end{proof}

For any subshift $W$, its transitivity pair relationship is closed.
So we may define a \emph{generalized transitivity} hierarchy 
similar to the TCPE hierarchy as follows.

\begin{definition}\label{def:gtr}
If $W$ is a subshift, define $T_1 \subseteq W^2$ to be its set of 
transitivity pairs.  Then for each $\alpha < \omega_1$ define 
$T_{\alpha+1}$ to be the transitive closure of $T_\alpha$ if $T_\alpha$ 
is closed, and to be the closure of $T_\alpha$ otherwise. 
For $\lambda$ a limit, define $T_\lambda = \cup_{\alpha<\lambda} T_\alpha$.
We say that $W$ is \emph{generalized transitive} if there is some 
$\alpha$ such that $T_\alpha = W^2$, in which case the least such 
$\alpha$ is called the \emph{generalized transitivity rank} of $W$.
\end{definition}

Observe that the shifts with transitivity rank 1 are exactly the 
transitive ones.  We have the following relationship between transitivity rank 
and TCPE rank.

\begin{theorem}
Suppose that $W$ is a subshift with a fully supported measure.  
Let $f$ be the operation of 
Theorem \ref{thm:pavlov-modified}.  Then $f(W)$ has TCPE 
if and only if $W$ is generalized transitive, and in this case the 
TCPE rank of $f(W)$ and the transitivity rank of $W$ coincide.
\end{theorem}
\begin{proof}
For any $T \subseteq W^2$, let $f(T)$ denote $\bigcup_{(x,x') \in T} f(x) \times f(x')$.
Let $E_\alpha$ denote the sets of the TCPE hierarchy on $f(W)$. 
We claim that for all $\alpha \in [1,\omega_1)$, we have $f(T_\alpha) = E_\alpha$.
The fact that $f(T_1) = E_1$ follows from Theorem \ref{thm:pavlov-modified}
part (4).  The limit step cannot cause any discrepancy.  In consideration of 
the successor step, suppose that $f(T_\alpha) = E_\alpha$ for some $\alpha$.

We claim $T_\alpha$ is closed if and only if $E_\alpha$ is closed. 
Observe $T_1$ is 
shift-invariant in the sense that $(x,x')\in T_1$ implies 
$O(x) \times O(x') \subseteq T_1$, and therefore each $T_\alpha$ is 
also shift-invariant in that sense (this shift-invariance is not destroyed 
by topological or transitive closures, or by limits).  Therefore, if $T_\alpha$
is closed, its complement is a union of sets of the form 
$U_v\times U_{v'}$, where
$v$ and $v'$ are patterns and
$$U_{v} := \{x: \text{ $v$ appears in $x$}\}.$$
Therefore the complement of $E_\alpha$ is a union of sets of the form 
$f(U_v) \times f(U_{v'})$.  Theorem \ref{thm:pavlov-modified} parts (1) and (5) 
imply that $f$ maps 
open shift-invariant sets to open shift-invariant sets.  Since the $U_v$ 
are shift-invariant, it follows that if $T_\alpha$ is closed, then so 
is $E_\alpha$.  On the other hand, if $T_\alpha$ is not closed, there are 
$(x,x') \not\in T_\alpha$ but $(x_n,x_n') \in T_\alpha$ with $(x_n,x_n')\rightarrow (x,x')$. 
Mapping these points all over to $f(W)$ using the same ribbon structure 
and same origin yields the analogous situation in $E_\alpha$.  This 
completes the proof of the claim that $T_\alpha$ is closed if and only if 
$E_\alpha$ is closed.

If $T_\alpha$ is closed, then it is easy to check that $(x,x')\in T_{\alpha+1}$ 
if and only if $f(x)\times f(x') \subseteq E_{\alpha+1}$.  On the other hand, 
if $T_\alpha$ is not closed, then an argument as above shows that 
$(x,x') \in T_{\alpha+1}$ implies $f(x) \times f(x') \subseteq E_{\alpha+1}$ 
and vice versa.
\end{proof}

Therefore it would suffice to show that
the SFT $Y_2$ defined in the previous section has a fully supported
measure and has generalized transitivity rank 3.
Unfortunately, it is possible that $Y_2$ does not have a fully supported 
measure.  However, the reasons are non-essential and 
we can modify the tiling construction to ensure a fully supported measure.

\subsection{An SFT with a fully supported measure and generalized transitivity rank 3}\label{sec:drs-trap}
We now describe why $Y_2$ may not have a fully supported 
measure, and how to fix this.
We would like to show that given an element $x \in Y_2$ 
and a pattern $v \in x$, that there is some $x' \in Y_2$ such 
that $v$ appears with positive density in $x'$.  For simplicity, 
let us consider a type $\infty$ element of $Y_2$ in which the 
$Y_1$ part uses only the symbol $\downarrow$.  
Consider its tiling structure.  
We could imagine a very unlucky pattern $v$ with the following properties:
\begin{itemize}
\item $v$ is a macrotile located in
in the computation part of its parent
\item The parts of the parent computation that $v$ sees are actually a 
rare combination that occurs in only one of all possible parent macrotiles
\item The uniquely implied parent macrotile also has these three properties.
\end{itemize}
Such unlucky $v$ may well exist, and if it exists it could occur at most 
one time in any configuration.  

To fix this we use a trick similar to one Durand and Romashchenko
used for a similar problem in \cite{DurandRomashchenko-TA}.  We make a 
``trap zone'' of size $2\times 2$ child macrotiles,
located in a part of the parent macrotile which is out of the computation 
zone and out of the way of the wires.  
Where exactly to put the trap zone can be efficiently computed 
in the same way as the wire layout is computed.  The idea 
is that any $2\times 2$ block of level-$i$ macrotiles that is 
permitted to appear at all, in any location, 
should be permitted to appear 
in the appropriately sized
trap zone.  To achieve this, we modify the direction of 
information flow at the eight level-$i$ macrocolors 
which form the boundary 
of the trap zone, so that all information at level 
$i$ is flowing out of the trap zone only.

We relax the
consistency requirements for tiles whose location is north, south 
east or west of one of the tiles in the trap zone.  
Note that if a tile has three of its 
location parts in agreement, the location of that tile is already
determined.  So if a tile knows itself (based on 3 location parts) 
to be a neighbor of the trap zone, then it will require only 
the consistency of the three parts of its color not touching the 
trap tile, and it will allow its trap-adjacent macrocolor to be
unrestricted.  Thus a trap neighbor $M$ will 
just observe whatever the adjacent trapped tile $M'$ chooses 
to hallucinate in all fields, including hallucinated 
``self-information of $M$'',
``corner messages sent by $M$'',
``diagonal neighbor information communicated by $M$''
and so on.
By reversing information flow in this way, 
any locally admissible $2\times 2$ block of 
tiles is permitted to appear in the trap zone.

Now, the trapped tiles will display accurate self-information
because they are internally consistent.\footnote{If 
a trapped tile thinks that it is a trap neighbor, it might display 
a false self-knowledge on one of its sides!  But its other 
side will have its true self-knowledge.  The information displayed 
by the group of trapped tiles is sufficent for the trap neighbors to 
correctly deduce which self-knowledge is correct 
if two versions are offered.}
The trap neighbors must examine the self-knowledge of 
the trapped tiles
to get some information:
\begin{itemize}
\item What sizes of square have appeared in the
trapped tile?  (The trap neighbors 
make sure the true parent has recorded the numbers
from the trapped tiles parameter tapes.)
\item What partial sides and corners appear in the trapped 
tile? (This is also readable from the parameter tape of 
trapped tiles.)
\end{itemize}

The eight trap neighbors and the four diagonal 
trap neighbors use a new 
\emph{trap information part} of their colors to 
pass information in a twelve-tile loop.  The passed information is:
\begin{itemize}
\item the combined self-information displayed by all four
trapped tiles, and
\item any side messages that trap neighbors would have 
wanted to pass through the trap zone
\end{itemize}
Based on this information, the trap neighbors are 
fully equipped to fill in all the missing size-checking functions of the 
trapped tiles.  If they are able to compute the 
size of a square whose corner is in the trapped 
tiles, they make sure the parent records that size.
If they are not able to compute the size of such a
square, they make sure the parent has recorded 
the deep coordinates of the partial side or corner in the trap zone.
If the parent cites a trapped tile as justification for 
a size, the trap neighbors check if this is warranted.
Finally, 
any messages that should have passed directly 
through the trapped tile are routed around it, and any 
messages that the trapped tile should have sent
out are sent out (with correct pixel coordinates, which 
the trapped tile would not have been able to compute by itself, 
since it believes itself to be in a different part of the parent).

The consistency checks for these modifications are all efficient to compute.
To summarize, 
let $Y_3$ be the SFT defined just as $Y_2$ was defined, but using 
the following colors and algorithm instead.

A level-$i$ macrocolor contains:
\begin{enumerate}
\item Location, wire and computation bits
\item Self-knowledge, neighbor-knowledge, diagonal-neighbor-knowledge, 
and parent-knowledge bits
\item Corner message sending and receiving bits
\item (NEW) Trap neighbor information circle bits
\begin{enumerate}
\item Eight self-knowledge fields, one for each side displayed by the trapped tiles.
These are used to pass the self-knowledge of the trapped tiles around to all trapped
neighbors.
\item Eight corner message sending fields, one for each trap neighbor.  These are used to inform all trap neighbors of any message that a trap neighbor wanted to send
through the trap zone (but was unable because information cannot flow into the trap zone).
\end{enumerate}
\end{enumerate}

The algorithm $A(t)$ does the following (changes italicized):
\begin{enumerate}
\item (Parsing)
\item (Consistency)
\begin{enumerate}
\item Check location, wire, and computation parts are consistent,
\emph{
except if 
three locations indicate this tile is a trap neighbor, in which case 
the trap-adjacent
location/wire/computation is permitted to be anything.}
\item Check consistency of self-, neighbor-, diagonal-neighbor-
and parent-knowledge across the four colors, \emph{except if this tile 
is a trap tile neighbor, in which case allow anything in the trap-adjacent color,
and deduce neighbor- and diagonal-neighbor-information 
involving trapped tiles from the trap neighbor information circle.}
\item \emph{If this tile is a trap neighbor, 
\begin{enumerate}
\item check that the self-information
displayed by the adjacent trapped tile also appears in the
appropriate field of the information circle bits
\item check that any message this tile wanted to send through 
the trap zone instead appears in the appropriate field 
of the information circle
\item check that the remaining fields of the information circle 
agree on both sides (that is, the information that other trap 
neighbors are re-routing is being accurately copied around).
\item Use the information circle to 
deduce the accurate self-information of all trapped tiles,
and to learn which messages the other trapped neighbors 
wanted to send through the trap zone.
\end{enumerate}}
\item Perform corner-message-passing functions. \emph{If this 
tile is a trap tile neighbor, do not directly
exchange messages with 
the trapped tile.  Instead, look to the information circle to see if any 
other trap neighbor wanted to send a message to this tile, 
and to see if any trapped tile should have sent a
message to this tile.}
\item Make sure the parent knows of any sizes deducible from
self-, neighbor- and diagonal-neighbor knowledge.
\item Make sure the parent accurately records or doesn't record
any partial corner or side from this tile, as appropriate. \emph{Trap
neighbors also do this on behalf of the trapped tiles.}
\item If the parent cites the location of this macrotile as justification 
for a size, make sure that is true and that the lexicographically 
least macrotile was cited.  \emph{Trap neighbors also check 
this on behalf of the trapped tiles.}
\item If the parent claims a partial corner or side that would intersect 
this macrotile, make sure that is true. \emph{Trap neighbors also 
check this on behalf of the trapped tiles.}
\item Check the parent-knowledge does not contain squares 
of too-different sizes.
\end{enumerate}
\item (Sync Levels)
\item If any of the above steps do not check out, halt.  Otherwise, run forever.
\end{enumerate}

Observe that $X_{0,1}$ is still a factor of $Y_3$.  That is because 
Proposition \ref{prop:up-accurate}, Lemma \ref{lem:adjacent}, 
and  Lemma \ref{lem:down-accurate} go through 
with only the following slight modifications.
\begin{enumerate}
\item We say that a macrotile $\bar t$ from $T$ is self-accurate 
at $u$ in $y$ if the top color of $T$ is self-accurate, and all the 
child macrotiles of $T$ are accurate \emph{except possibly the 
trapped tiles, who are only required to be self-accurate}.
\item In Proposition \ref{prop:up-accurate}, the 
only change is the more convoluted process 
by which the self-information of trapped tiles 
gets to the parent.
\item There is no modification to Lemma \ref{lem:adjacent} 
because the trapped tiles do not occur at the edge of their parent.
\item In Lemma \ref{lem:down-accurate}, 
we build superpositions where the trapped tiles happen to hallucinate 
the truth, so we may continue setting all parts of all macrocolors
to accurate values.  
\end{enumerate}
We leave the easy details to the reader.

\begin{lemma}\label{lem:fsm}
The SFT $Y_3$ has a fully supported measure.
\end{lemma}
\begin{proof}
Suppose that $w$ is a pattern that appears in some $z = (x,y) \in Y_3$.  Since 
every pattern is eventually contained in a $2\times 2$ block of macrotiles 
at some level, without loss of generality we can assume that $w$ is a 
$2\times 2$ block of macrotiles.  Let $h:Y_3\rightarrow Y_1$ be the 
obvious factor map.
We choose a finite number $n_0$ 
as follows.  If $\type(h(x)) = n$ or $\type(h(x)) = (n,n+1)$, let $n_0=n$.  
If $\type(h(x)) = \infty$, then $h(w)$ is either part of an interior of a square, 
or contains 
some combination of partial corners and/or partial sides.  Though these 
partial squares are infinite in $x$, it is consistent with $h(w)$ that they 
be completed into large finite squares.  In this case, 
let $n_0$ be large enough that 
$(n_0,n_0+1) \in \type(h(w))$, and also large enough that no square of 
size $n_0$ could have more than a partial corner or side appear in 
any region equal in size to the combined responsibility zones 
of the tiles of $w$.

Now, fix a non-exceptional tiling structure.  The pattern $w$ is the right shape to fit in 
a trap zone at some level.  Trap zones the size of $w$ appear with positive 
density.  Let $N$ be the number guaranteed by Lemma \ref{lem:extension} 
for $n_0$.  Let $M$ be a number large enough that every $M\times M$ 
square contains a $w$-sized trap zone that is fairly central: the edge of 
the combined responsibility zones of the trapped tiles 
is not closer than $N$ pixels away from the edge of the $M\times M$ square.
Build $y \in Y_1$ as follows.  First lay $n_0$-squares in horizontal stripes, 
so that the blank space between the stripes is $M$ pixels tall.  Then lay 
$n_0$ squares in vertical strips so that the blank space between the strips 
is $M$ pixels wide.  (It maybe necessary to play with the spacing of 
$n_0$-squares, adjacent vs. a one-pixel gap, in order to achieve this,
but $M$ is bigger than $N$, so it is possible.) 
Then put $h(w')$ into one fairly central trap zone in each $M\times M$ 
blank region, where $w'$ is the restriction of $y$ to the combined 
responsibility zones of the tiles of $w$.  
Then use Lemma \ref{lem:extension} to fill in $n_0$-squares 
and $(n_0+1)$-squares through the entire remaining space in each
region.  The result 
is an element $y' \in Y_1$.

Now we argue that there is an $x' \in T$ such that $z :=(x',y') \in Y_3$
 contains $w$ with positive density.  Let $j_0$ denote the level 
 of the four macrotiles of $w$.  Let $u \subseteq \mathbb Z^2$ be 
 an $L_i \times L_i$ region, where $i>j_0$, and where $u$ corresponds 
 to a single macrotile in the fixed tiling structure.  We claim that for every 
 macrocolor $\bar c$ that is self-accurate at $u$ for $y'$, there 
 is a valid $T$-pattern $\bar t$ which is self-accurate at $u$ for $y'$, 
 has top-level color $\bar c$,
 and furthermore $\bar t$ has $w$ in every fairly 
 central trap zone.
 
 The base case occurs when $i = j_0+1$.  If $u$ does not contain 
 a fairly central trap zone, we apply Lemma \ref{lem:down-accurate} 
 and are done.  So assume $u$ does contain a fairly central 
 trap zone.  We are trying to construct
 a level-$i$ macrotile and we will do so by selecting appropriate 
 level-$j_0$ child macrotiles $M$ and arguing that they fit together.  There are three
 cases.
 \begin{enumerate}
 \item If $M$ is neither trapped nor a trap neighbor, we choose a 
 $T$-pattern $\bar t_M$ for $M$ exactly as in Lemma \ref{lem:down-accurate},
 choosing accurate values for all fields, and referring to $\bar c$ for the unique 
 choices of wire and computation bits.
 \item If $M$ is trapped, let $\bar t_M$ be the corresponding tile of $w$.  By 
 Proposition \ref{prop:up-accurate}, $\bar t_M$ is self-accurate at its location 
 in $y$.  Since self-accuracy depends only on the responsibility zone, and 
 $y'$ copied $y$ on the entire responsibility zone of the trapped tiles, 
 $\bar t_M$ is self-accurate at its location in $y'$.
 \item If $M$ is a trap neighbor, set its colors to match the corresponding 
 trapped tile of $w$ on its trap size.  Set its standard colors on the other 
 three sides to accurate values.  Set its information circle colors to 
 accurately reflect the self-information that all trapped tiles are displaying,
 and to accurately reflect the messages that should be sent through the 
 trap zone on all sides.  Since these colors are accurate, apply 
 Lemma \ref{lem:down-accurate} to produce the macrotile $\bar t_M$.
 \end{enumerate}
 The colors above are chosen so that all tiles are self-accurate 
 and all adjacent tiles have top-level
 colors that 
 can go next to each other.  Therefore, by Lemma \ref{lem:adjacent}, all the
child macrotiles created above can fit together to form a level-$i$ 
macrotile $\bar t$ with the right properties.
 
 The inductive case is now identical with the proof of Lemma \ref{lem:down-accurate}.
It now follows by compactness that there is an infinite $T$-tiling $x'$ 
that is accurate at each pixel in $y'$ and copies $w$ in each fairly central 
trap zone.  Since the fairly central trap zones occur with positive density,
we are done.
\end{proof}

\begin{lemma}\label{lem:gtr3}
The SFT $Y_3$ has generalized transitivity rank 3.
\end{lemma}
\begin{proof}
We claim that $x,x'$ in $Y_3$ are a transitivity pair exactly when 
$h(x),h(x')$ in $X$ are an entropy-or-equal pair.  This is determined 
by type (see Lemma \ref{lem:ep-via-type}).  When $h(x),h(x')$ are 
not an entropy-or-equal pair, observe it always happens for a particularly 
strong reason: $h(x)$ and $h(x')$ were not even a transitivity pair.  
Therefore, $x,x'$ cannot be a transitivity pair either.  On the other 
hand, if $h(x)$ and $h(x')$ have compatible type, then for any 
patterns $w\in x$ and $w' \in x'$, there is an $n_0$ such that 
$(n_0,n_0+1) \in \type(h_1(w)) \cap \type(h_1(w'))$.  As in the previous 
lemma, it suffices to consider $w$ and $w'$ which are the same
size and which each consist of four 
macrotiles in a $2\times 2$ arrangement.  As in the previous lemma, 
construct $y \in Y_1$ by fixing a tiling structure, laying $n_0$-squares 
in an $M\times M$ grid pattern for large enough $M$, 
placing $h_1(w)$ or $h_1(w')$ (free choice) in fairly 
central trap zones in each $M\times M$ region, and then filling in 
the rest of $y$ and the computations just as in the previous lemma.
Apply compactness and the result is an element of $Y_3$ in which 
both $w$ and $w'$ appear.
\end{proof}

\begin{theorem}
There is a $\mathbb Z^2$-SFT with TCPE rank 3.
\end{theorem}
\begin{proof}
Our example is $f(Y_3)$, where $f$ is the map from Theorem \ref{thm:pavlov-modified}.
Apply that theorem, Lemma \ref{lem:fsm} and Lemma \ref{lem:gtr3}.
\end{proof}

\section{A family of $\mathbb Z^2$ shifts}\label{sec:rankalpha}

To generalize the previous construction to all computable ordinals, 
we define a transformation whose input is a tree $T \subseteq \omega^{<\omega}$
and whose output is a subshift $X_T \subseteq 2^{\mathbb Z^2}$, 
whose TCPE status and rank are controlled by properties of $T$.
It will be technically more convenient to use trees 
$T \subseteq \Omega^{<\omega},$ where $\Omega = \{(n,n+1) : n \in \omega\}$.

\subsection{Definition of the family of shifts}

\begin{definition}
Given square patterns $A,B \in 2^{[0,m)^2}$, and 
subpatterns $C \preceq A, D \preceq B$,
let $R_{A,B,C,D}$ denote the set of 
restrictions which say: in every configuration in which both a 
 $C$ and a $D$ 
appear, both an $A$ and a $B$ must appear.  
Furthermore every occurrence of $A$ or $B$ must have another 
occurrence of $A$ or $B$ directly north, south, east and west.
\end{definition}

Informally, when restrictions $R_{A,B,C,D}$ are applied, the permitted configurations 
fall into two categories.  Configurations in which $C$ and $D$ do not coexist are 
unrestricted.  But in configurations where $C$ and $D$ do coexist, the 
 configuration can be understood as an (affine) configuration on $\{A,B\}^{\mathbb Z^2}$.

\begin{definition}
For each $\sigma \in \Omega^{<\omega}$, define $A_\sigma,B_\sigma$ to be
the following patterns in $2^{[0,m)^2}$ (where $m$ depends on $\sigma$).
Let $A_\lambda = 0$ and $B_\lambda = 1$.  If $A_\sigma$ and $B_\sigma$ 
have already been defined, let $B_{\sigma\concat(n,n+1)}$ be an $(n+1)$-square 
on alphabet $\{A_\sigma,B_\sigma\}$, and let $A_{\sigma\concat(n,n+1)}$ 
be an $n$-square on the same alphabet, plus a row of $B_\sigma$ along the top 
and along the right side (to make it the same size as $B_{\sigma\concat(n,n+1)}$.
$$A_{\sigma\concat(n,n+1)} = \begin{array}{cccccc} B_\sigma & \dots & \dots & \dots &\dots & B_\sigma\\
A_\sigma & \dots & \dots & \dots & A_\sigma & \vdots\\
\vdots & B_\sigma & \dots & B_\sigma & \vdots & \vdots \\
\vdots & \vdots & \ddots & \vdots & \vdots & \vdots \\
\vdots & B_\sigma & \dots & B_\sigma & \vdots & \vdots \\
A_\sigma & \dots & \dots & \dots & A_\sigma & B_\sigma\end{array},
B_{\sigma\concat(n,n+1)} = \begin{array}{cccccc} A_\sigma & \dots & \dots & \dots & \dots & A_\sigma\\
\vdots & B_\sigma & \dots & \dots & B_\sigma & \vdots\\
\vdots & \vdots & \ddots & & \vdots & \vdots\\
\vdots & \vdots &  & \ddots & \vdots & \vdots\\
\vdots & B_\sigma & \dots & \dots & B_\sigma & \vdots\\
A_\sigma & \dots & \dots & \dots & \dots & A_\sigma\end{array}
$$
\end{definition}

\begin{definition}  Definition of $R_\sigma$ and $F_\sigma$.
\begin{enumerate}
\item No restrictions are imposed by $R_\lambda$.  For $\sigma \in \Omega^{<\omega}$ of 
length at least 1, write $\sigma = \rho\concat(n,n+1)$, let $R_{\sigma}$ denote 
$R_{A_\rho,B_\rho, C, B_\rho}$, where $C$ is the $n$-square on alphabet $\{A_\rho, B_\rho\}$. 
\item Let $F_\sigma$ denote $F_{A_\sigma,B_\sigma}$, with the $F$ restrictions 
as in Definition \ref{def:FAB}.
\end{enumerate}
\end{definition}

\begin{definition}
For any tree $T \subseteq \Omega^{<\omega}$, define $X_T$ to be the subshift 
defined by forbidding $R_\sigma$ and $F_\sigma$
for each $\sigma \in T$.
\end{definition}

Observe that if $T = \{\lambda\}$, then $R_\lambda$ makes no 
restrictions and $F_\lambda = F_{0,1}$ ensures that $X_T$ is equal to the subshift $X$ 
considered in the previous section.

The intuition behind the definition can be understood by considering the 
simple case $T = \{\lambda, (1,2)\}$.  We will speak about this case informally 
in order to motivate the subsequent arguments. 
The entropy pair relations for 
$X_T$ are pictured in Figure \ref{fig:simpleT}\footnote{
The figure is accurate other than a technical caveat
about alignment; technically there are 16 parallel widgets 
between type 1 and type 2.}.  
\begin{figure}[hbpt]\label{fig:simpleT}
\includegraphics[width=\textwidth]{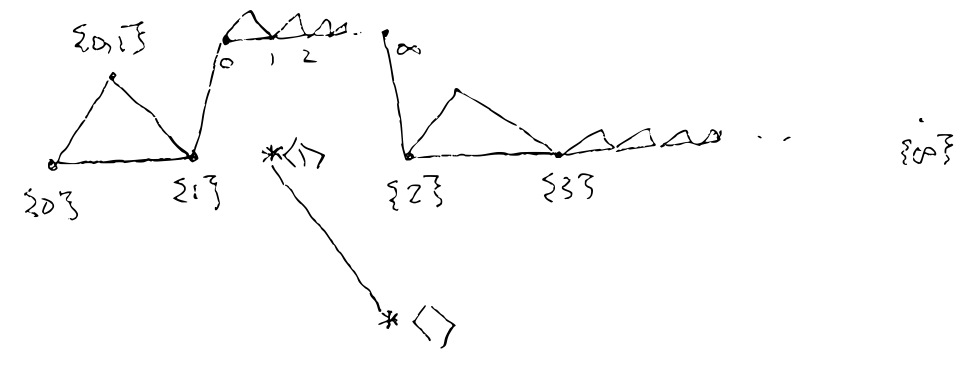}
\end{figure}
Because $\lambda \in T$, 
$X_T$ is a subshift of $X$.  The inclusion of $R_{(1,2)}$ removes many 
elements of $X$ which had type $(1,2)$, and the remaining elements 
of type $(1,2)$, the pattern of 1-squares 
and 2-squares is very regular, these elements can be understood 
as configurations on the alphabet 
$A_{(1,2)}, B_{(1,2)}$.   Therefore it makes sense to apply $F_{(1,2)}$ 
to these configurations.  This causes a fracturing of the $(1,2)$ type into 
subtypes: 
$$(1,2)\concat0, (1,2)\concat(0,1), (1,2)\concat1, 
\dots, (1,2)\concat\infty,$$ similar to how $X$ was fractured.  
In $X$, an element of type 1 was an entropy pair with an element of 
type 2, but in $X_T$, no element of type 1 is an entropy pair with 
any element of type 2, for the same reason that $0^{\mathbb Z^2}$ (type 0)
and $1^{\mathbb Z^2}$ (type $\infty$) were not entropy pairs in $X$ 
(the connecting configurations would need to contain both 
0-squares on alphabet $\{A_{(1,2)}, B_{(1,2)}\}$ and $n$-squares on 
the same alphabet, where $n$ is large).

Note, however, that among configurations with subtype $(1,2)\concat 0$,
there is one which is 
is just an infinite array of $A_{(1,2)}$, i.e. an orderly array of 1-squares on the 
original alphabet.  Therefore this configuration is also 
in some sense type 1.  Similarly, there is a configuration with subtype $(1,2)\concat\infty$ 
which is just an infinite array of $B_{(1,2)}$, i.e. consists of tightly packed 
2-squares on the original alphabet.  Therefore this configuration 
is in some sense type 2.  These connections provide the reason why $X_T$ still 
has TCPE (in fact, TCPE rank 4).    The first step $E_1$ is as in 
Figure \ref{fig:simpleT}.  The equivalence relation $E_2$ has three classes, 
which could be named ``types less than $(1,2)\concat \infty$'',  
``finite types $\geq (1,2)\concat\infty$'' and ``$\infty$''.
The closed relation $E_3$ provides connections from the first equivalence 
class to the special $(1,2)\concat\infty$ element, and from the second 
equivalence class to the $\infty$ element.  And $E_4$ is everything.

\subsection{Types and the Hausdorff derivative}

We assign each $x \in X_T$ a \emph{type} and an \emph{alignment}.
The type of $x$ is a string in $(\omega \cup \Omega \cup \{\infty\})^{\leq \omega}$
which depends on $x$ and $T$.  The \emph{alignment} is an element 
of $(\mathbb N^2)^{\leq \omega}$.  

Given $\sigma \in \Omega^{<\omega}$, let $j_\sigma\in \Nat$ be 
the unique number such 
that $A_\sigma, B_\sigma \in 2^{[0,j_\sigma)^2}$

\begin{definition}
Given $\sigma \in \Omega^{<\omega}$, we say that $x \in 2^{\mathbb Z^2}$ 
is \emph{$\sigma$-regular} if there is some $g \in [0,j_\sigma)^2$ such 
that for all $h \in (j_\sigma\mathbb Z)^2$, $x\uhr ([0,j_\sigma)^2+h+g)$ is equal to 
either $A_\sigma$ or $B_\sigma$.  If such $g$ exists it is unique and 
we say $g$ is the \emph{$\sigma$-alignment} of $x$.
\end{definition}

That is, $x$ is $\sigma$-regular if $x$ can be parsed as a tiling on 
alphabet $\{A_\sigma,B_\sigma\}$.  Observe that if $x$ is $\sigma$-regular,
then $x$ is $\tau$-regular for every $\tau \preceq \sigma$, and $x$ 
is not $\tau$-regular for any $\tau$ that is incomparable with $\sigma$.

\begin{definition}\label{defn:types}
For a tree $T \subseteq \Omega^{<\omega}$ and $x \in X_T$, define 
$\type_T(x)$ as follows.  Let $\rho \in T \cup [T]$ 
be longest such that $x$ is $\rho$-regular.  
\begin{itemize}
\item If $\rho \in [T]$, define 
$\type_T(x) = \rho$.  
\item If $\rho \in T$, consider $x$ as a configuration 
on alphabet $\{A_\rho,B_\rho\}$ and let $t \in \omega \cup \Omega \cup \{\infty\}$ 
be its unique type in the sense of Definition \ref{defn:type}.  Then define 
$\type_T(x) = \rho\concat t$.
\end{itemize}
\end{definition}

If $\sigma \in T$ and $x \in X_T$ is $\sigma$-regular, then if $x'$ 
denotes the shift of $x$ by any amount not in $(j_\sigma\mathbb Z)^2$,
 in general 
$x$ and $x'$ will not be an entropy pair, because their alignment 
is off.  In practice this does not cause any difficulties because 
the alignments get progressively washed out at limit stages of 
the TCPE ranking process. So with the caveat that alignments are 
much less important than the types, but still necessary for 
the reader who wants all details, we define them as follows.

\begin{definition}
For a tree $T \subseteq \Omega^{<\omega}$ and $x \in X_T$, define 
$\alignment_T(x)$ as follows.  Let $\rho \in T \cup [T]$ 
be longest such that $x$ is $\rho$-regular.  Define $\alignment_T(x)$ 
to be the unique string $\nu \in (\Nat^2)^{\leq \omega}$ 
such that $|\nu| = |\rho|$ and for all $k < |\nu|$, 
$\nu(k)$ is the $(\rho\uhr k)$-alignment of $x$. 
\end{definition}

Every type is an element of $(\omega \cup \Omega \cup \infty)^{\leq \omega}$.  
Which of these elements actually appear as types?  

\begin{definition}
Given $T \in \Omega^{<\omega}$, define 
$$T^+ = \{\sigma\concat a : \sigma\concat a \not\in T, \sigma \in T, a \in \omega \cup \Omega \cup \{\infty\}\}.$$
\end{definition}

\begin{prop}
For $T \in \Omega^{<\omega}$, we have $\{\type_T(x) : x \in X_T\} = T^+ \cup [T]$.
\end{prop}
\begin{proof}
If $\type_T(x)$ is infinite, then $\type_T(x) \in [T]$.  It is easy to construct $x$ of 
given infinite type.  If $\type_T(x) = \rho\concat t$ as in Definition \ref{defn:types}, 
suppose for contradiction that 
$\rho\concat t \in T$.  Then $t = (n,n+1)$ for some $n$, so $x$ contains 
both an $n$-square and an $(n,n+1)$-square on alphabet $\{A_\rho,B_\rho\}$.  Since $\rho\concat t \in T$, 
the restrictions $R_{\rho\concat t}$ are present and apply to $x$.  Therefore $x$ is $\rho\concat t$-aligned,
contradicting the initial choice of $\rho$.  On the other hand it is easy to construct an $x \in X_T$ with 
$\type_T(x)  = \rho\concat t \in T^+$, by constructing any configuration of type $t$ on alphabet $\{A_\rho,B_\rho\}$.
\end{proof}

We 
put the following natural linear order on $\omega \cup \Omega \cup \infty$:
$$0 < (0,1) < 1 < (1,2) < 2 < \dots < \infty.$$

This induces a lexicographic order on types.

\begin{definition}
Let $L_T$ be the linear order $(\{\type_T(x) : x \in X_T\}, \leq_{\text{lex}})$.
\end{definition}

Finally, it will be convenient to consider only those $T$ 
where the branches of $T$ are spread out a little bit.
Define $2\Omega = \{(2n, 2n+1) : n \in \omega\}$.  If $T \in (2\Omega)^{<\omega}$, 
we still have 
$\{\type_T(x) : x \in X_T\} \subseteq (\omega \cup \Omega \cup \{\infty\})^{\leq \omega}$
because the ``odd'' elements of $\Omega$ can still appear at the end of $\type_T(x)$.


Recall the \emph{Hausdorff derivative} is the following operation on linear 
orders.  Given a linear order $L$, define an equivalence relation 
$H$ on $L$  by $aHb$ if there are only finitely many $c$ with 
$a \leq c \leq b$ or $b \leq c \leq a$.  The equivalence classes 
are well-behaved with respect to the order.  Output the linear order $L'$
whose elements are the $H$-equivalence classes.

The $\alpha$-th Hausdorff derivative is then defined by transfinite 
iteration of the Hausdorff derivative.  (At limit stages, take the 
union of all equivalence relations defined so far.)

Recall that a linear order is \emph{scattered} if $\mathbb Q$ 
cannot be order-embedded into it.  It is well-known that 
repeated applications of the Hausdorff derivative 
stabilize to the equivalence relation $L^2$ if and only if 
$L$ is scattered.  The \emph{Hausdorff rank} of a scattered
linear order is the least $\alpha$ at which this occurs.

\subsection{TCPE ranks of these shifts}

For $T \in (2\Omega)^{<\omega}$, 
we show that $X_{T}$ has TCPE if and only if 
$\{\type_{T}(x) : x \in X_{T}\}$ is scattered when considered 
as a linear order with the lexicographical ordering.  Furthermore, 
there is a precise level-by-level correspondence 
between the TCPE rank of $X_{T}$ and the Hausdorff rank 
of $L_T$.

\begin{theorem}\label{thm:tcpe-hausdorff}  Given $T \subseteq (2\Omega)^{<\omega}$, the shift
$X_{T}$ has TCPE if and only if $L_T$ is scattered.  Furthermore, 
if $X_{T}$ has TCPE rank $\alpha$ and $L_T$ has Hausdorff rank $\beta$, then
$\alpha \in \{2\beta -1, 2\beta\}$.
\end{theorem}

For the rest of this section, $T\subseteq (2\Omega)^{<\omega}$ is some tree.
The lexicographic order on types can be lifted 
to $X_T$. 

\begin{definition}
Define
a total pre-order on $X_T$ by 
$x \lex y$ if $\type_T(x) \leq_{\text{lex}} \type_T(y)$.
\end{definition}

The lifted pre-order relates to the topology of $X_T$ in the following sense.
\begin{lemma}\label{lem:limit}
If $x_n \rightarrow x \in X_T$ in its topology, and $\rho_1 <_\tlex \type_T(x) <_\tlex \rho_2$
for some types $\rho_1$ and $\rho_2$, then for sufficiently large $m$, we have 
$\rho_1 \lex \type_T(x_m) \lex \rho_2$.
\end{lemma}
\begin{proof}  First, let $\rho\in T$ be longest such that $\rho \prec \rho_1$ 
and $\rho \prec \type_T(x)$.  Let $t_1$ and $t$ be the next symbols of each type, 
that is $\rho\concat t_1 \preceq \rho_1$ and $\rho\concat t \preceq \type_T(x)$. 
We know that $t_1 < t$.  If $t = (n,n+1)$, then $x$ contains both an 
$n$-square and an $(n+1)$-square on alphabet $\{A_\rho, B_\rho\}$, 
so for sufficiently large $m$, each $x_m$ also contains these features.  
So for sufficiently large $m$, we have $\rho\concat t \prec \type_T(x_m)$,
and thus $\rho_1 <_\tlex \type_T(x_m)$.  If $t = n$ for some $n \in \omega$, 
then $n\neq 0$ (otherwise $\rho_1$ could not be less), and for sufficiently 
large $m$, $x_m$ must also contain an $n$-square on alphabet $\{A_\rho,B_\rho\}$.
Therefore, we eventually have $\rho\concat (n-1,n) \lex \type_T(x_m)$.  
If $\rho_1 = \rho\concat (n-1,n)$, we are done with this case.  If $\rho_1$ 
is longer, that means that $\rho\concat (n-1,n) \in T$.  We know that 
$x$ is not $\rho\concat(n-1,n)$-aligned (if it were, we would have chosen 
a longer $\rho$).  Therefore, there is some pattern in $x$ that contains $n$-squares 
which are arranged in a way incompatible with such an alignment.  
For sufficiently large $m$, the configurations $x_m$ also contain such an arrangement. 
Thus for sufficiently large $m$, these $x_m$ cannot contain an $(n-1)$-square,
so $\rho\concat (n-1,n) <_\tlex \type_T(x_m)$.    Finally, if $t = \infty$, then $t_1$ 
is finite, but $x$ contains arbitrarily large blocks of $B_\rho$.  For sufficiently 
large $m$, each $x_m$ contains blocks of $B_\rho$ large enough to 
ensure that $\rho\concat t_1 <_\tlex \type_T(x_m)$.

The argument for the upper bound is similar.  Let $\rho\in T$ be longest such that 
$\rho \prec \rho_2$ and $\rho\prec \type_T(x)$, and let $t, t_2$ be such that 
$\rho\concat t \prec \type_T(x)$ and $\rho\concat t_2 \prec \rho_2$.  If 
$t \in \Omega$ the argument is as above.  If $t = n\in \omega$ then now $n=0$ is possible, 
but in any case the argument is as above.  And $t=\infty$ is 
not possible because then there is nothing that $\rho_2$ could be.
\end{proof}

Let $H_\alpha \subseteq X_T^2$ be the associated liftings of the Hausdorff derivatives $L_T^\alpha$ 
of $L_T$.  That is, if we let $[\rho]_\alpha$ denote the equivalence class of the type $\rho$ 
in $L_T^\alpha$, we have
$$H_\alpha = \{(x,y) : \type_T(x) \in [\type_T(y)]_\alpha\}.$$

In particular, $H_1$ is the set of all $(x,y)$ for which there are only finitely many 
types between $\type_T(x)$ and $\type_T(y)$.

\begin{lemma}\label{lem:type-separation}
For $x,y \in X_T$, if there are infinitely many types between $\type_T(x)$ and $\type_T(y)$,
then $x$ and $y$ 
are not a transitivity pair.
\end{lemma}
\begin{proof}  We can assume that $x <_\tlex y$.  Supposing that $x$ and $y$ 
are a transitivity pair, we show that there are not infinitely many types between them.
Let $\rho\in T$ be longest such that $\rho \prec \type_T(x)$ and $\rho \prec \type_T(y)$. 
Let $t_x, t_y \in \omega\cup \Omega \cup \{\infty\}$ such that 
$\rho\concat t_x \prec \type_T(x)$ and $\rho\concat t_y \prec \type_T(y)$.  
Then $t_x < t_y$.  Thus $t_x \neq \infty$, so $x$ contains a copy of $A_\rho$, 
and $t_y \neq 0$, so $y$ contains a copy of $B_\rho$.  Therefore, 
any configuration in which large patterns from $x$ and $y$ coexist 
contains both an $A_\rho$ and a $B_\rho$, and is thus subject to 
the restrictions $R_\rho$ and $F_\rho$.  Relative to the alphabet 
$\{A_\rho, B_\rho\}$, if $x$ contains an $n$-square, then 
$y$ cannot contain any square larger than $n+1$ (including infinite 
partial squares).  

So if $t_x = n$, then $t_y \in \{(n,n+1), n+1\}$.  If $(n,n+1)\not\in T$, 
then there are only finitely many types between $x$ and $y$, 
so we are done.  If $(n,n+1) \in T$, then $y$ 
cannot contain $(n+1)$-square.  If it did, 
$x$ and $y$ would not be a transitivity pair, because an element with sufficiently 
large patterns from $x$ contains a pattern which has $n$-squares
but is not $\rho\concat(n,n+1)$-aligned; if such an element also contains 
$(n+1)$-square, it would be forbidden.  Therefore, if $(n,n+1) \in T$, the only 
possibility is $t_y = (n,n+1)$, but $y$ contains only $n$-squares, 
that is, $y$ is an array of $A_{\rho\concat(n,n+1)}$, 
and $\type_T(y) = \rho\concat(n,n+1)\concat 0$.  Thus $x$ and $y$ have successor 
types in this case.  

Since $t_x \neq \infty$, the last case to consider is if $t_x = (n,n+1)$ for some $n$.
Then the only option for $t_y$ is $t_y = n+1$.  If $(n,n+1)\not\in T$, we are done, 
so suppose that $(n,n+1)\in T$.  Then $x$ cannot contain an $n$-square.  If it 
did, then any configuration containing large patterns from $x$ and $y$ would 
contain an $n$-square from $x$, an $(n+1)$-square from $y$, and an arrangment 
of $(n+1)$-squares from $y$ that is incompatible with a $\rho\concat(n,n+1)$-alignment; 
this is forbidden.  Since $x$ contains no $n$-squares, it must be that $x$ is 
an infinite array of $B_{\rho\concat(n,n+1)}$, and thus $\type_T(x) = \rho\concat(n,n+1)\concat\infty$.
Therefore, $x$ and $y$ have successor types.
\end{proof}

It follows that if $x$ and $y$ do not satisfy $xH_1 y$, then $x$ and $y$ are not an entropy pair, 
and furthermore there is no finite chain $x = z_0,z_1,\dots,z_n=y$ in which each $z_i,z_{i+1}$
is an entropy pair (between some $z_i$ and $z_{i+1}$ there will be infinitely many types).
This shows that $E_2 \subseteq H_1$.  

The relations $H_\beta$ have the following nice property.

\begin{definition}
A equivalence relation $F \subseteq X_T^2$ is \emph{interval-like} if for all $x \lex y \lex z$ in $X_T$, 
if $xFz$ then $xFy$ and $yFz$.
\end{definition}

Note that if $F$ is interval-like, there is a natural ordering on the equivalence classes of $F$ defined by 
$[x]_F < [y]_F$ if and only if for all $x' \in [x]_F$ and all $y' \in [y]_F$, we have $x' <_{\text{lex}} y'$.

To show the first half of Theorem \ref{thm:tcpe-hausdorff}, we take advantage of the 
fact that the subshift $X_T$ was designed to have topological connectedness
roughly corresponding to the order topology on $L_T$.  Since the topological 
closure operation cannot do much more than connect elements from 
successor equivalence classes,\footnote{Technically a connection is 
possible from $[x]_\beta$ to $[y]_\beta$
if there is at most one equivalence 
classes between $[x]_\beta$ and $[y]_\beta$, but this is good enough.}
the combined double-operation of the TCPE process cannot connect things 
faster than the Hausdorff derivative.

\begin{lemma}\label{lem:ep-slowerthan-h}
For all $\beta \geq 1$, we have $E_{2\beta} \subseteq H_\beta$.
\end{lemma}
\begin{proof} By induction.
The case $\beta=1$ was dealt with above. The limit case is clear.  For the successor 
case, suppose that $E_{2\beta} \subseteq H_\beta$.  Since $E_{2\beta+2}$ 
is obtained by taking two closure operations on $E_{2\beta}$, we are done if 
we can show that applying the same two closure operations to $H_\beta$ 
yields a subset of $H_{\beta+1}$.

Let $H_\beta'$ denote the topological closure of $H_\beta$ in $X_T^2$.  
If $(x,y) \in H_\beta'$, then there is a sequence $(x_n,y_n)_{n\in \omega}$ 
of elements of $H_\beta$
whose limit is $(x,y)$.  
Without loss of generality, assume that $[x]_\beta <  [y]_\beta$.   
We would like to conclude that there are only finitely many $H_\beta$-equivalence
classes between $[x]_\beta$ and $[y]_\beta$.  Suppose for contradiction 
that there are $z,w \in X_T$ such that $[x]_\beta < [z]_\beta < [w]_\beta < [y]_\beta$.
Because $H_\beta$ is interval-like, we have $x <_\tlex z$, so 
by Lemma \ref{lem:limit}, eventually $x_n \lex z$.  Similarly, eventually 
we have $w \lex y_n $.  However, this is a contradiction, because 
each $(x_n,y_n) \in H_\beta$ and $H_\beta$ is interval-like.

Let $H_\beta''$ denote the symmetric and transitive closure of $H_\beta'$. 
If $(x,y) \in H_\beta''$ then there is a sequence $x=z_1,\dots,z_n=y$ 
such that each $(z_i,z_{i+1}) \in H_\beta'$.  Therefore there are at most 
finitely many $H_\beta$-equivalence classes between $[x]_\beta$ and $[y]_\beta$.
So $(x,y) \in H_{\beta+1}$.
\end{proof}

Now we turn our attention to showing that the TCPE process connects things 
as quickly as the Hausdorff derivative allows, subject to a small caveat about alignment.

Take as an example the case $T = \{\lambda, (1,2)\}$ which was discussed 
at the start of this section.  It is easy to construct configurations $x,y \in X_T$ 
where for example each may have type $(1,2)\concat 3$, and yet $x$ and $y$ 
are not an entropy pair, because the $\{A_{(1,2)},B_{(1,2)}\}$ grid of $x$ may 
be off by a single pixel shift from the corresponding grid of $y$.  Thus patterns 
of $x$ and $y$ cannot be independently swapped.  Now, $x$ and $y$ 
are a transitivity pair, so we could solve this issue by putting $X_T$ on fabric
as in Theorem \ref{thm:pavlov-modified}.  However, we have not yet added in the 
computation part, and if we make $X_T$ all wiggly before adding the computation, 
the required algorithm becomes very messy.   
It is simpler instead to notice that all we need 
to prove the theorem is to show that the TCPE derivation process 
already proceeds fast 
enough when its domain is restricted to a subset of $X_T$ on which 
all alignments are compatible.  If the derivation process goes fast 
enough on each restricted domain, its progress
cannot get slower when all the domains are considered together.

\begin{definition}
An \emph{alignment family} is a function 
$a: T \rightarrow \mathbb N^2$ such  
that for every $\rho \in T$, there is an $x \in X_T$ such that 
$\rho \prec \type_T(x)$ and $a(\tau)$ is the $\tau$-alignment of $x$
for all $\tau \preceq \rho$.
\end{definition}
If you imagine building such $a$ in layers starting with defining $a(\lambda)$ (which has to be (0,0)), 
then defining $a$ on all strings of length 1, etc, 
 at each $\rho$ the values of $a(\tau)$ for $\tau\prec \rho$, as well as 
 the value at $\rho(|\rho|-1)$, place some simple and deterministic 
 restrictions on what $a(\rho)$ 
 can be, and other than that you have free choice of $a(\rho)$. 
 The important point, which 
 the reader can verify, is that for every $x \in X_T$, there is an alignment family $a$
such that for all $\rho \in T$ with $\rho \prec \type_T(x)$, we have 
$a(\rho)$ is the $\rho$-alignment of $x$.  Just fill in $a$ in layers, 
but when there is a choice at some $\rho \prec \type_T(x)$, copy 
the alignment of $x$.

Given an alignment family $a$, let $X_T^a$ denote the subset of $X_T$ 
consisting of all those $x$ for which the $\rho$-alignment of $x$ 
is $a(\rho)$ for every $\rho$ for which $x$ is $\rho$-regular.  We see that 
$X_T = \cup_a X_T^a$, however any given $x$ is in multiple of these sets, 
and some $x$, such as for example $1^{\mathbb Z^2}$, are in all of them.
For each ordinal $\alpha$ and alignment family $a$, let $H_\alpha^a$ 
and $E_\alpha^a$ denote
denote the restrictions of $H_\alpha$ and $E_\alpha$ to $X_T^a$ respectively.

\begin{lemma}\label{lem:ep-successors}
Suppose $x,y \in X_T$ such that 
$x$ and $y$ have the same $\rho$-alignment for any $\rho$ 
such that $x$ and $y$ are both $\rho$-regular.  Then 
\begin{enumerate}
\item If $\type_T(x) = \type_T(y)$ then $x$ and $y$ are an entropy-or-equal pair.
\item If there are no types strictly between $\type_T(x)$ and $\type_T(y)$,
then there are $x' \equiv_\tlex x$ and $y' \equiv_\tlex y$, with the same 
alignments as $x$ and $y$, such that $x'$ and $y'$ are 
an entropy-or-equal pair.
\end{enumerate}
\end{lemma}
\begin{proof}
Without loss of generality assume that $x \lex y$.  Let $\rho\in T \cup [T]$ be longest such 
that $\rho\prec\type_T(x)$ and $\rho\prec\type_T(y)$.  If $\rho \in T$
then $x$ and $y$ can both be understood 
as configurations on the alphabet $\{A_\rho,B_\rho\}$.  In the following cases we 
will refer several times to $n$-squares and appeal several times to Lemma \ref{lem:ep-via-type}.
In all cases, these references should be understood with respect to the alphabet 
$\{A_\rho,B_\rho\}$.

{\bf Case 1.} Suppose $\type_T(x) = \rho\concat n$ where $n \in \omega$.  If $\type_T(y)$
 is either $\rho\concat n$ or $\rho\concat(n,n+1)$, then $x$ and $y$ are an entropy pair 
 by Lemma \ref{lem:ep-via-type}.  If $\rho\concat(n,n+1) \in T$ then the 
 $L_T$-successor of $\rho\concat n$ 
 is $\rho\concat(n,n+1)\concat 0$.  If this is $\type_T(y)$, let $y'$ be a 
 configuration which consists of tiling the plane with $A_{\rho\concat(n,n+1)}$, using $A_\rho$
 and $B_\rho$ in the appropriate alignment.  This configuration counts as type 0 relative to the alphabet 
 $\{A_{\rho\concat(n,n+1)},B_{\rho\concat(n,n+1)}\}$ because squares are allowed to 
 be flush next to each other, so this is the tightest possible packing of 0-squares
 relative to $\{A_{\rho\concat(n,n+1)},B_{\rho\concat(n,n+1)}\}$.  Therefore, $y'\equiv_\tlex y$.
  (This is why we must permit squares to touch, and it is
  the only place where that distinction is needed).  Now $y'$, like $x$, 
  contains only $n$-squares; it is just a coincidence that $y'$ is $\rho\concat (n,n+1)$-regular, 
 and the entropy pairhood of $x$ and $y'$ can be realized by elements of 
 type $\rho\concat n$.
 
 {\bf Case 2.}  Suppose $\type_T(x) = \rho\concat \infty$.  Since $\rho \prec \type_T(y)$, 
 it must be that $\type_T(y)$ is also $\rho\concat \infty$.  The proof 
 of Lemma \ref{lem:ep-via-type} for the case where both configurations have 
 type $\infty$ shows that $x$ and $y$ will be an entropy-or-equal 
 pair if there are infinitely many $t$ such that $\rho\concat(n,n+1)\not\in T$. 
 This condition is satisfied because $T \subseteq (2\Omega)^{<\omega}$.  This is 
 one of two places where it is used that $T \subseteq (2\Omega)^{<\omega}$ rather than 
 $\Omega^{<\omega}$.
 
 {\bf Case 3.} Suppose $\type_T(x) = \rho\concat(n,n+1)$, where $(n,n+1)\not\in T$.  Then 
 $\type_T(y)$ is one of $\rho\concat(n,n+1)$ or $\rho\concat(n+1)$, and $x$ and $y$ 
 are an entropy pair by Lemma \ref{lem:ep-via-type}.
 
 {\bf Case 4.} Suppose $\rho\concat(n,n+1) \prec \type_T(x)$, where $\rho\concat(n,n+1)\in T$. 
 By the choice of $\rho$ we know that $y$ is not $\rho$-aligned, so the only way to
 get $x$ and $y$ to be $\lex$-successors is if $\type_T(x) = \rho\concat(n,n+1)\concat\infty$ 
 and $\type_T(y) = \rho\concat (n+1)$.  Let $x'$ be a configuration which 
 consists of tiling the plane with $B_{\rho\concat(n,n+1)}$, using $A_\rho$ and $B_\rho$ 
 with the proper alignment.  Then 
 $x' \equiv_\tlex x$, and $x'$ contains only $(t+1)$-squares; it is just a coincidence 
 that $x'$ is $\rho\concat(n,n+1)$-regular, and the entropy pairhood of $x'$ 
 and $y$ can be realized by elements of type $\rho\concat(n+1)$.
 
 {\bf Case 5.} Suppose that $\type_T(x) = \rho$ where $\rho \in [T]$.  Then $\rho$ 
 has no $\lex$-successor, so $\type_T(y) = \rho$ as well.  Suppose that $u$ is a finite 
 pattern in $x$ and $v$ is a finite pattern in $y$.  Then for some $\tau \prec \rho$, 
 $u$ and $v$ are each contained in a $2\times 2$ block of the alphabet $\{A_\tau, B_\tau\}$, 
 so assume that $u$ and $v$ are each just a $2\times 2$ block on that alphabet.
 Based on just a $2\times 2$ block, the potential types of $u$ and $v$ relative to alphabet 
 $\{A_\tau, B_\tau\}$ are almost unrestricted (the only restriction is if all four are $B_\tau$, 
 in which case there must be a $n$-square with $n \geq 2$ relative to alphabet 
 $\{A_\tau, B_\tau\}$).  So by Lemma \ref{lem:extension}, we may make 
 independent choices of $u$ and $v$ on a set of small enough positive density 
 and fill in the gaps to produce a
 configuration of type $\tau\concat(2n+1,2n+2)$ for any $n \geq 1$.
 We have no alignment restrictions (beyond sticking to the alphabet $\{A_\tau, B_\tau\}$)
 when building a configuration of this type because 
 $T \subseteq(2\Omega)^{<\omega}$ (the second of two places where this 
 assumption on $T$ is used.)  Therefore, $x$ and $y$ are an entropy-or-equal pair.
 \end{proof}

Therefore, we may conclude that for each $a$, we have $H_1^a = E_2^a$.
(We have already shown previously that $E_2 \subseteq H_1$.)

\begin{lemma}\label{lem:connectedness}
Fix an alignment family $a$.  
If $F$ is an interval-like equivalence relation on $X_T^a$ with $H_1^a \subseteq F$, 
and $[x]_F,[y]_F$
are a successor pair of $F$-equivalence classes, then the topological 
closure of $F$ contains a pair $(x',y')$ with $x' \in [x]_F$ and $y' \in [y]_F$.
\end{lemma}
\begin{proof}
We may assume $[x]_F < [y]_F$.  

{\bf Case 1.}  Suppose there is a longest $\rho \in T$ such that there exist 
$x' \in [x]_F$ and $y' \in [y]_F$ with
 $\rho \prec \type_T(x')$, and $\rho \prec \type_T(y')$.  Let $t \in \Omega\cup\omega$ 
 be largest such that for some $x' \in [x]_F$, we have $\rho\concat t \prec \type_T(x'),$ 
 if such $t$ exists.  
 
 If no such $t$ exists, then $[x]_F$ has elements of type 
 $\rho\concat n$ for arbitrarily large $n \in \omega$ and $[y]_F$ has all the $a$-aligned elements 
 of type $\rho\concat \infty$ (since $[y]_F$ does have something whose type starts 
 with $\rho$, and $\infty$ is all that is left).  For each sufficiently large $n$, let $x_n \in [x]_T$ 
 be a configuration of type $\rho\concat n$.  Every limit point $y'$ of this sequence 
 has type $\rho\concat \infty$.  So $(x,y')$ is in the closure of $F$ via a subsequence 
 of $(x,x_n)_{n\in\omega}$.
 
 If $t$ exists, then $[x]_F$ contains a $\lex$-greatest type, which is $\rho\concat t$ if 
 $\rho\concat t \not\in T$, and which is $\rho\concat t\concat \infty$ if $\rho\concat t \in t$. 
 Letting $s$ be the successor of $t \in \Omega \cup \omega$,
  $[y]_F$ contains a $\lex$-least type, which is $\rho\concat s$ if $\rho\concat s \not\in T$, 
  and $\rho\concat s \concat 0$ if $\rho\concat s \in T$.  But that
  implies that some $x' \in [x]_F$ and $y' \in [y]_F$ are a $\lex$-successor pair,
  contradicting that $H_1^a \subseteq F$.
 
 {\bf Case 2.} There is some infinite $\rho \in [T]$ such that for each $\tau \prec \rho$, 
 there are $x' \in [x]_F$ and $y' \in [y]_F$ which are $\tau$-aligned.  The elements 
  of type $\rho$ are either in $[x]_T$ or in $[y]_T$.  If they are in $[x]_T$ then 
  for all sufficiently large $n$, there is $y_n \in [y]_T$ with 
  $\type_T(y_n) = (\rho\uhr n)\concat \infty$.  Every limit point $x'$ of this sequence 
  has type $\rho$, because the larger and larger alphabets $\{A_{\rho \uhr n}, B_{\rho\uhr n}\}$
  are adopted cofinally in this sequence.  Similarly, if the type $\rho$ elements 
  are in $[y]_T$, we can define $x_n$ to be a configuration of type $(\rho\uhr n)\concat 0$,
  and any limit point $y'$ has type $\rho$ for the same reason.
 \end{proof}

\begin{proof}[Proof of Theorem \ref{thm:tcpe-hausdorff}]
By Lemma \ref{lem:ep-slowerthan-h}, we know that for all
$\beta \geq 1$, we have $H_\beta \supseteq E_{2\beta}$.
This shows that if $L_T$ has Hausdorff rank $\beta$, then 
the TCPE rank of $X_T$ is at least $2\beta - 1$, 
and that if $L_T$ is not scattered, then $X_T$ does not have TCPE.

By Lemma \ref{lem:ep-slowerthan-h}, we also know that for all
$\beta \geq 1$ and any alignment family $a$, we have $H_\beta^a \supseteq E_{2\beta}^a$.
In fact, these two sets are equal, which we show by induction. 
The base case is above and the limit case is clear.  Assuming that $H_\beta^a = E_{2\beta}^a$, 
by Lemma \ref{lem:connectedness}, we see that each successor pair 
of equivalence classes of $H_\beta^a$ gets a connection between at least 
one pair of representatives in the topological closure of $H_\beta^a$.  
Therefore, taking symmetric and transitive closure, we find that 
$H_{\beta+1}^a \subseteq E_{2\beta+2}^a$.

It now follows that if $H_\beta = L_T^2$, 
then $E_{2\beta}^a = (X_T^a)^2$ for all $a$.
The configuration $1^{\mathbb Z^2}$ is an element of every $X_T^a$,
and $E_{2\beta}$ is an equivalence relation.  Therefore, 
$E_{2\beta} = X_T^2$.  
\end{proof}

\subsection{Enforcing the restrictions with sofic computation}

In this section we show that if $T \in (2\Omega)^{<\omega}$ is 
a computable tree, then $X_T$ is sofic, via an SFT extension 
which has the same generalized transitivity rank as $X_T$.

We have to make a couple of modifications to the 
algorithm developed in Sections \ref{sec:drs-basic} and 
\ref{sec:drs-trap}.  Here are the major updates required:
\begin{enumerate}
\item The parameter tape of each macrotile should now keep track 
of the largest $\rho\in T$ for which the underlying configuration
appears to be $\rho$-regular, and the corresponding 
$\rho$-alignment relative to the macrotile, and check that 
other neighbors agree about this $\rho$ and this alignment.
Then the construction can proceed as in Section \ref{sec:rank3},
keeping track of sizes of squares on the alphabet $\{A_\rho,B_\rho\}$.
\item However, we cannot simply
superimpose symbols of $\Lambda$ onto occurrences of
 $A_\rho$ and $B_\rho$ for all $\rho$ (since these patterns 
 are arbitrarily large, it does not even make sense).  
 Instead we simulate the square-enforcing 
 function of the alphabet $\Lambda$ by 
 having each macrotile who thinks she is inside an $A_\rho$ 
 pattern guess about which symbol of $\Lambda$ 
 should be superimposed.  The friendly neighbor 
 communications can be slightly expanded to make sure everyone 
 inside of a given $A_\rho$ agrees about the symbol, and to 
 make sure all the $2\times 2$ restrictions of $F_\rho$ 
 are satisfied.
  \end{enumerate}

Of course, point (1) above is a slight lie.  If the alignment $\rho$ 
of a particular configuration is very long, or contains very large 
numbers, then it is likely that small macrotiles will not have a 
tape long enough to record such $\rho$.  Instead, they 
should record an initial segment of $\rho$ that is long enough, 
and leave it to their parent, grand-parent, and so on to 
lengthen the alignment as warranted.

Another caveat about point (1) is that we want to make 
sure that entropy pair relations of $X_T$ are preserved
as transitivity pair relations in the SFT that is currently 
under construction.  So for example, if $T = \{\lambda, (1,2)\}$,
we do want there to be a transitivity pair relationship 
between an element $x$ of type 1 and the special element $y$ of 
type $(1,2)\concat 0$ which consists of a regular grid 
of the macrosymbol $A_{(1,2)}$.  The macrotiles superimposed 
on $x$ will be aware that $x$ contains only $1$-squares 
but is not $(1,2)$-aligned, while 
the macrotiles superimposed on $y$ will be aware that 
$y$ is $(1,2)$-aligned but does not contain any $B_{(1,2)}$.  
It should be consistent 
for both kinds of macrotile to exist in a single configuration 
which combines large patterns from $x$ and $y$.  We 
can achieve this if we stick to the convention that a macrotile 
keeps track of what is happening in its locality only.  
If two neighbor macrotiles have observed different but 
consistent things in their individual locality, both sets of 
observations 
will be assimilated by the parents of these macrotiles, 
in the same way that a parent who has one child seeing 
only $n$-squares and another seeing only $(n+1)$-squares
can record both of those facts.  In our example with 
$T= \{\lambda, (1,2)\}$, a parent macrotile whose 
locality contains both large patterns from $x$ and 
large patterns from $y$ would have recorded that 
its locality is $\lambda$-regular, that there are 1-squares 
in that locality, and that the locality is neither 
$(0,1)$-regular nor $(1,2)$-regular (the former witnessed 
by a child in the $y$ part and the latter witnessed 
by a child in the $x$ part).

In the above paragraph we have left the notion of locality 
deliberately vague.  That is because the macrotiles 
and the macrosymbols are unlikely to line up exactly, 
and so if a macrotile is to know about the macrosymbols
which are in its vicinity, 
it will necessarily know about things beyond 
the boundary of the usual responsibility zone.  This will 
not cause a problem, nor will it be necessary to 
give a precise definition to the term ``locality''.
However, it will be necessary to be precise about 
how much of $\rho$ a given macrotile should 
record (in the case where the type of the 
configuration is, or locally appears to be, 
much longer than what could be written on the 
macrotile's tape).

Clearly, a macrotile should know the largest $\rho$ such that 
the pattern in its usual 
responsibility zone is $\rho$-aligned.  Imagine some 
parent macrotile which knows this information.
If $A_\rho$ and $B_\rho$ 
are much smaller than this parent macrotile, then the parent 
cannot do
the work of figuring out 
whether the configuration is so far satisfying $F_\rho$, 
 because the parent cannot know which macrosymbol appears at each location 
(too much information compared to the tape size of the parent).
Therefore, we need the children to primarily do this work
(although if there are any remaining partial corners or sides made 
out of macrosymbols $\{A_\rho,B_\rho\}$,
we want the children to report this 
to the parent).  Assuming the children have success 
at enforcing $F_\rho$, then the parent macrotile and 
its neighbors can use the passed-up information to 
proceed exactly as in Section \ref{sec:drs-trap}, either 
recording sizes reported by children, or passing messages to 
figure out what kind of larger squares are being made with
alphabet $\{A_\rho, B_\rho\}$. 

How will the macrotile's children figure out whether the 
configuration is so far satisfying $F_\rho$?  They need to 
make a guess, for each appearance of the macrosymbol 
$A_\rho$ or $B_\rho$, 
which symbol of $\Lambda$ (the auxiliary alphabet 
used in Section \ref{sec:drs-basic}) should be superimposed
upon each macrosymbol.  Then, by looking at $2\times 2$ 
blocks of macrosymbols, they should forbid any combinations 
where the superimposed $\Lambda$ symbols are 
inappropriately placed.  An issue seems to arise: what if the 
macrosymbols are very large compared to the children?  
(If they are small, we can punt the problem to the grandchildren, 
so let us assume they are larger than the children.)
How could the children be expected to know that it is time 
for them to guess a symbol of $\Lambda$, and 
where one macrosymbol from $\{A_\rho,B_\rho\}$ ends 
and the next begins?
This information cannot actually be deduced by looking 
at the part of the configuration which is in the responsibility 
zone of the child.  The children have to guess: what is $\rho$, 
how is the alphabet $\{A_\rho,B_\rho\}$ aligned, which of these 
two macrosymbols am I contained in?  Only then can the child 
also guess about the symbol from $\Lambda$.
Thus we see that the desired connection between information 
at the child level and parent level will be possible only if the child 
correctly guesses $\rho$ and its alignment.  This means in 
general that a macrotile needs to know the longest $\rho$ 
such that the pattern in the responsibility zone of its \emph{parent}
appears to be $\rho$-aligned.  Knowing this is sufficient as well.

Now we describe in more detail the algorithm that will be run on 
the macrotiles.  But we do not give all details, trusting that the reader 
who has made it this far would be more hindered than helped by 
an overabundance of technical elaboration.   

We first give a summary 
of the algorithm to be performed on all macrotiles, followed by remarks 
which give slightly more details about how to achieve any step
for which there might be a question.  
Let $Y_T$ be the SFT defined by superimposing the following 
computation onto $Y_1$ (where $Y_1$ is the shift on alphabet $\Lambda$ 
described in Section \ref{sec:rank3}.)

On input $p,c$:
\begin{enumerate}
\item Data format check.
\begin{enumerate}
\item The parameter tape contains
\begin{itemize}
\item A level number $i$ and starting number $i_0$
\item A string $\rho \in\omega^{<\omega}$ for which the size of $A_\rho$ is less than $L_{i+1}$
\item Deep coordinates, relative to the parent, indicating where a single macrosymbol from $\{A_\rho,B_\rho\}$
has a lower left corner.
\item A couple bits to indicate which of $A_\rho$ and $B_\rho$ have been sighted, and if so,
example(s) of where.
\item An assertion of whether $F_\rho$ has been followed or not (and if not, where the violation occurred).
\item Up to two sizes (of squares observed in alphabet $\{A_\rho,B_\rho\}$).
\item If any sizes appear above, deep coordinates for some corners of squares of those sizes
(witnessing that they exist and that they have no regularity)
\item Up to four deep coordinates for corners or sides of partial squares on alphabet $\{A_\rho,B_\rho\}$,
\item If the size of $A_\rho$ is larger than $L_i$, 
up to four symbols of $\Lambda$, our guesses for what is superimposed on up to four 
$A_{\rho}$ or $B_\rho$ patterns in which we may be participating.
\end{itemize}
\item The colors contain
\begin{itemize}
\item Location, machine and wire bits
\item A copy of the parent's parameter tape
\item Side message-passing bits
\item Friendly neighbor parameter display bits
\item Diagonal neighbor parameter display bits
\item Trap neighbor message-passing bits
\end{itemize}
\end{enumerate}
\item Expanding tileset construction, including the ``trap zone'' feature discussed in Section \ref{sec:drs-trap}.
\item Size checking.  If the parent asserts anything that relates to me 
(for example, that I contain some $n$-square on some alphabet with a 
corner at such-and-such location), make sure 
that assertion is accurate.  
Let $\nu$ denote the string recorded by my parent (i.e. $\nu$ is my parent's $\rho$).
If $\rho$ is a proper initial segment of $\nu$, and 
the size of $A_\nu$ is larger than $L_{i+1}$, my
parent's parameter tape uniquely determines my entire configuration,
so no further checks are necessary.  If $\rho$ is a proper 
initial segment of $\nu$ and the size of $A_\nu$ is smaller 
than $L_{i+1}$, then I should have said $\nu$ instead of $\rho$; if that 
happens, kill the tiling.
If $\nu \preceq \rho$, that means that $\nu$ is longest such that all my 
siblings are $\nu$-aligned.  If the parent indicates that $F_\nu$ is violated, 
no further checks are necessary.  If the parent indicates that 
$F_\nu$ has been followed, we split into two cases.
\begin{enumerate}
\item Case 1: If the size of $A_\nu$ 
is greater than or equal to the size of me ($L_i$), then I intersect up to four 
symbols from alphabet $\{A_\nu,B_\nu\}$.  For each such macrosymbol, I have 
guessed a symbol from $\Lambda$ to superimpose.  (If $\nu = \rho$, I guessed 
this directly, if $\nu \prec \rho$, my guess for alphabet $\{A_\rho,B_\rho\}$ 
uniquely determines a guess for alphabet $\{A_\nu,B_\nu\}$.)
If I have guessed a border-corner 
symbol of $\Lambda$, and if I contain the outermost pixel 
associated to that corner of 
that $A_\nu$, I 
must send and receive messages for that corner, compute the size of a 
square on alphabet $\{A_\nu,B_\nu\}$ when I receive matching messages, 
and require my parent to record either the size or the partial corner (as appropriate),
with full details just as in Section \ref{sec:drs-basic}.
\item Case 2:  If the size of $A_\nu$ 
is less than $L_i$, I can assume that my children have already 
parsed the $\{A_\nu,B_\nu\}$ macrosymbols and informed me of any partial 
corners or sides on that alphabet which I may have.  
(If $\nu\prec\rho$, the information on my tape allows me to uniquely 
determine what partial corners or sides on alphabet $\{A_\nu,B_\nu\}$ are 
located in my vicinity, and also allows me to determine if any complete 
$n$-squares on that alphabet are in my vicinity.) Proceed as in Section \ref{sec:drs-basic}.
\end{enumerate}
\item Friendly and diagonal neighbor steps.  Using my own parameter tape plus those of eight neighbors,
\begin{itemize}
\item Certify that my parent's claims are consistent with what is on all my neighbors' parameter tapes
\item Let $\nu$ be longest such that all eight neighbors are $\nu$-aligned. If 
the size of $A_\nu$ is greater than or equal to $L_i$, 
collect all the neighbors' guesses about what symbols of $\Lambda$ to superimpose 
on their macrosymbols $A_\nu$ and $B_\nu$  (in some cases this may have to be inferred 
from their guesses about larger macrosymbols).  If I and a neighbor both intersect 
the same macrosymbol, make sure our guesses are the same (if they are not the same,
kill the tiling).  If the $3 \times 3$ 
block which I am viewing contains a 4-way 
boundary of macrotiles on alphabet $\{A_\nu,B_\nu\}$, check that the $2\times 2$ 
restrictions on $\Lambda$ are satisfied.  If they are not satisfied do not kill 
the tiling, but do make 
sure the parent records that $F_\nu$ has been violated.
\item If the size of $A_\nu$ is less than $L_i$, we can assume our children 
have already checked for $F_\nu$ compliance.  If our tape says that $F_\nu$ 
was violated, nothing more to do.  If our tape says that $F_\nu$ was followed, 
use the parameter tapes of non-sibling neighbors to find any sizes 
of squares on alphabet $\{A_\nu,B_\nu\}$ that straddle the parent 
boundary, and report these to the parent just as in Section \ref{sec:drs-basic}.
\end{itemize}
\item\label{step:tn} Trap neighbors: Read the parameter tapes of the trapped tiles, 
communicate what is on those tapes to all 12 trap neighbors, and reproduce 
all the missing functions (message-passing and $F_\nu$ compliance checking) 
which the trapped tiles should have performed.
\item Compute facts about $T$.  Halt if you ever see either of these situations:
\begin{itemize}
\item If $\rho \in T$, but $F_\rho$ has not been followed, or
\item If $n$- and $(n+1)$-squares on alphabet $\{A_\rho,B_\rho\}$ 
have been reported, and $\rho\concat(n,n+1) \in T$, 
but $\rho\concat(n,n+1)$-regularity has been shown to fail.
\end{itemize}
\end{enumerate}

All of the objects being computed on have size at most $\log(L_{i+2})$, 
and the operations are all polynomial time operations except for the 
final step of computing facts about $T$.  Therefore, all but the last 
step of the algorithm takes $\poly(\log L_{i+2})$ time, which
is asymptotically much less than $N_i/2$.  It follows 
that we can define the tileset to start from an $i_0$ large enough 
that each macrotile will finish steps (1)-(5) within half of its available time,
leaving the other half for $T$ computation.
A sufficiently large macrotile superimposed on a configuration of 
impermissible type will be large enough to have learned the type 
and large enough to compute the $T$-facts which witness 
that the type is impermissible.  At that point, the impermissible 
type will be forbidden.  Therefore, $X_T$ is a factor of $Y_T$.

Let $T \in (2\Omega)^{<\omega}$ be a computable tree and let $Y_T$ 
be defined as above.  Let $h:Y_T\rightarrow X_T$ be the obvious 
factor map.  Then we have the following lemmas.

\begin{lemma}
$Y_T$ has a fully supported measure.
\end{lemma}
\begin{proof}
Let $w$ be any pattern that appears in an element $y$ of $Y_T$.  Without 
loss of generality, we can assume that $w$ consists of a $2\times 2$
array of macrotiles.  Let $\rho = \type_T(y)$.  We pick an alphabet 
$\{A_\sigma,B_\sigma\}$ and a $t\in\omega \cup \Omega$ as follows.  The goal 
is to end up with a finite type $\sigma\concat t$ 
that is consistent with $h(w)$.
If $\rho$ is finite, 
let $\sigma$ be all but the last symbol of $\rho$;
then $\rho = \sigma \concat t$ for some $t \in \omega \cup \Omega \cup \{\infty\}$.
If $t$ is finite, we are done.  If $t = \infty$, let $t\in \omega$ be a number large 
enough that it is consistent that $h(w)$ has type $t$ 
relative to alphabet $\{A_\sigma, B_\sigma\}$.  If $\rho$ is infinite, 
then let $\sigma \prec \rho$ be long enough that $w$ contains parts 
of no more than four macrosymbols on alphabet $\{A_\sigma,B_\sigma\}$.
The parameter tapes in $w$ have made $\Lambda$-guesses for these symbols.
One possibility is that all four symbols are $A_\sigma$ and the 
$\Lambda$-guesses imply that these four symbols form a 0-square;
 in that case, let $t = 0$.  If this possibility does not occur, then the 
 $\Lambda$-guesses are consistent with any finite $t$ that is sufficiently 
 large, so we may set $t= n$ for some large $n$.
Observe that in all cases we have arrived 
at a finite type $\sigma\concat t\not\in T$ 
which is consistent with $h(w)$ and with the information 
written on the four parameter tapes of $w$.

Relative to the alphabet $\{A_\sigma,B_\sigma\}$,
we can now proceed almost exactly as in Lemma \ref{lem:fsm}.  
The only additional detail to consider is the way in which 
the macrosymbol grid intersects $w$.  If we fix a tiling structure 
and then place a copy of $h(w)$ in one of the trap zones
at the appropriate level, this determines where the boundaries 
of the macrosymbols must lie throughout the entire configuration.
The trap zone of a different, arbitrarily chosen 
macrotile may not be able to accommodate a copy of $h(w)$ 
because the macrosymbol boundaries 
may intersect each trap zone in a different 
way.  However, there are only finitely many ways that the macrosymbol 
boundaries can intersect the trap zones, and each way occurs 
with positive density (in fact periodically).  Therefore, the argument of 
Lemma \ref{lem:fsm} can be repeated if we just choose $M$ large 
enough to ensure not only a sufficiently-central trap zone, 
but a sufficiently-central one with the right macrosymbol boundaries.
\end{proof}

\begin{lemma} If $y_n \rightarrow y \in Y_T$ in its topology, 
and $\rho_1 <_\tlex \type_T(h(y)) <_\tlex \rho_2$
for some types $\rho_1$ and $\rho_2$, then for sufficiently large $n$, we have 
$\rho_1 \lex \type_T(h(y_n)) \lex \rho_2$.
\end{lemma}
\begin{proof} Follows immediately from Lemma \ref{lem:limit}.\end{proof}

\begin{lemma}
For $y_1,y_2 \in Y_T$, if there are infinitely many types between 
$\type_T(h(y_1))$ and $\type_T(h(y_2))$, then $y_1$ and $y_2$ 
are not a transitivity pair.
\end{lemma}
\begin{proof}
Follows immediately from Lemma \ref{lem:type-separation}.
\end{proof}

Let $H_\alpha \subseteq Y_T^2$ denote the pull-back of the Hausdorff 
equivalence relations on $L_T$.  Recall that $T_\beta \subseteq Y_T^2$
refer to the relations in the generalized transitivity hierarchy 
of $Y_T$ (Definition \ref{def:gtr}).  

\begin{lemma}
For all $\beta \geq 1$, we have $T_{2\beta} \subseteq H_\beta$.
\end{lemma}
\begin{proof}
Identical to the proof of Lemma \ref{lem:ep-slowerthan-h}.
\end{proof}

\begin{lemma}
For $y_1,y_2 \in Y_T$, let 
$\rho_1 = \type_T(h(y_1))$ and $\rho_2 = \type_T(h(y_2))$.
Then
\begin{enumerate}
\item If $\rho_1 = \rho_2$, then $y_1$ and $y_2$ are a transitivity 
pair.  
\item If there are no types strictly between $\rho_1$ 
and $\rho_2$, then there exists a transitivity pair $(y_1',y_2')$
such that $\type_T(h(y_1')) = \rho_1$ and $\type_T(h(y_2')) = \rho_2$.
\end{enumerate}
\end{lemma}
\begin{proof}
This proof will be an addendum to Lemma \ref{lem:ep-successors}.
To deal with all cases at once, let us assume that $\rho_1 \lex \rho_2$,
and if 
$\rho_1 = \sigma\concat n$ and $\rho_2 =  \sigma\concat(n,n+1)\concat 0$, 
then $h(y_2)$ is just an infinite array of $A_{\sigma\concat(n,n+1)}$. 
Similarly let us assume that if $\rho_1 = \sigma\concat(n,n+1)\concat \infty$ 
and $\rho_2 = \sigma\concat(n+1)$, then $h(y_1)$ is just an infinite 
array of $B_{\sigma\concat(n,n+1)}$.  Now let us show that $y_1$ and $y_2$ 
are a transitivity pair.

Let $w_1$ and $w_2$ be patterns of $y_1$ and $y_2$ consisting of
$2\times 2$ arrays of macrotiles of the same size.  By the 
arguments of Lemma \ref{lem:ep-successors}, there exists a finite 
type $\sigma\concat t \not\in T$ such that this type 
is compatible with both $h(w_1)$ and $h(w_2)$.  Furthermore, 
in that Lemma it was essentially shown that for any two locations 
far enough apart, there is an element of $X_T$ with type 
$\sigma\concat t$ such that the patterns
$h(w_1)$ and $h(w_2)$ appear at those 
locations (provided the locations agree about the alignment 
of the macroalphabet $\{A_\sigma,B_\sigma\}$).

Fix a macrotile grid and place $h(w_1)$ in a trap zone.  
This determines the $\{A_\sigma,B_\sigma\}$ macrosymbol 
boundaries in the entire configuration.
It is likely that in $w_2$, the macrosymbol boundaries intersect 
the macrotiles in a different way than they did in $w_1$.  
We need to find a trap zone in which the macrosymbol boundaries 
occur in the same way as they do in $w_2$.  Here we use a 
primality trick.  The pixel size of a macrotile is always a power of 2. 
Whereas, the pixel size of a macrosymbol is $\prod_j (n_j+3)$, 
where $\sigma = (n_0,n_0+1)\cdots(n_k,n_k+1)$.  Since $\sigma \in T$, 
each $n_j$ is even.  Therefore, pixel size of a macrotile 
and the pixel size of a macrosymbol are relatively prime.   The 
distance between one trap zone and the next is exactly one macrotile. 
It follows that every possible way for macrosymbol boundaries 
to intersect the trap zones occurs with positive density.
Therefore, there is a trap zone, sufficiently far away, 
where $w_2$ fits.  Put $h(w_2)$ there.  Fill in the rest of the 
$Y_1$ part of the configuration, producing an element of type 
$\sigma \concat t$.  Then fill up the colors for the computation 
part of the configuration, copying $w_1$ or $w_2$ in the 
target zones, just as in Lemma \ref{lem:fsm}.
\end{proof}

This shows that $H_1 = T_2$.  Finally, we can finish the analysis 
with the analog of Lemma \ref{lem:connectedness}, but it 
is simpler because there is no need to deal with alignment families.

\begin{lemma}\label{lem:connectedness2}
If $F$ is an interval-like equivalence relation on $Y_T$ with $H_1 \subseteq F$, 
and $[x]_F,[y]_F$
are a successor pair of $F$-equivalence classes, then the topological 
closure of $F$ contains a pair $(x',y')$ with $x' \in [x]_F$ and $y' \in [y]_F$.
\end{lemma}
\begin{proof}
Same as the proof of Lemma \ref{lem:connectedness}.
\end{proof}

It follows that for all $\beta$, we have $H_\beta = T_{2\beta}$.  Therefore, 
we may conclude the main theorem of this section.

\begin{theorem}
For any computable tree $T\subseteq(2\Omega)^{<\omega}$, the
$\mathbb Z^2$-SFT $Y_T$ is generalized transitive if and only if $L_T$ is scattered.
Furthermore, if $Y_T$ has generalized transitivity rank $\alpha$ and 
$L_T$ has Hausdorff rank $\beta$, then $\alpha \in \{2\beta-1,2\beta\}$.\qed
\end{theorem}

\subsection{Main results}

To give examples of $\mathbb Z^2$-SFTs of various ranks, it is useful to have a 
class of trees for which it is easy to see the Hausdorff rank of $L_T$.  The 
next definition gives such a class.

\begin{definition}
A \emph{fat tree} is a tree $T \subseteq \omega^{<\omega}$ such that for every 
$\sigma \in T$ with well-founded rank $\alpha$ and every $\beta < \alpha$, there are 
infinitely many $n$ 
such that $\sigma\concat n\in T$ and $\sigma\concat n$ has rank $\beta$.  We 
allow $\alpha$ or $\beta$ to be $\infty$ in case $T$ is ill-founded, and declare 
$\infty < \infty$ to be true.
\end{definition}

Given any tree $T$, there is an easy procedure to turn it into a fat tree.  
For example, $T \times \omega^{<\omega}$ is fat, where 
$(\sigma, \tau) \in S\times T$ if $\sigma$ and $\tau$ have the same 
length, $\sigma \in S$ and $\tau \in T$.  Of course, $S\times T$ is computably 
isomorphic to a tree on $\omega^{<\omega}$, and it is immaterial whether 
$T \subseteq \omega^{<\omega}$ or $T\subseteq \Omega^{<\omega}$ 
or $T \subseteq (2\Omega)^{<\omega}$ or anything else.

\begin{prop}
If $T\subseteq (2\Omega)^{<\omega}$ is well-founded and fat with rank $\alpha$, 
then the Hausdorff rank of $L_T$ 
is $\alpha + 1$.
\end{prop}
\begin{proof}

We claim that for all $\rho \in T$, all $n \in \omega$ and all $\alpha$, that 
$[\rho\concat n]_\alpha = [\rho\concat (n+1)]_\alpha$ if and only if 
 $r_T(\rho\concat(n,n+1)) < \alpha$.  This is proved by induction. If $\alpha = 1$, 
 the distinction is only whether $\rho\concat(n,n+1) \in T$, and the conclusion 
 is clear from the definition of $L_T$.  
 
 If $\alpha$ is a limit, then $[\rho\concat n]_\alpha = [\rho\concat(n+1)]_\alpha$ 
 if and only if the same is true for some $\beta < \alpha$, if and only if $r_T(\rho\concat(n,n+1))<\beta$ 
 for some $\beta<\alpha$, if and only if $r_T(\rho\concat(n,n+1))< \alpha$.
 
 Finally if $\alpha=\beta+1$ suppose first that $r_T(\rho\concat(n,n+1)) \leq \beta$.  
 Sicne $\rho\concat n$ and $\rho\concat(n,n+1)\concat 0$ are successors 
 in $L_T$, they are in the same class in $L_T^\beta$.  Similarly, 
 $[\rho\concat(n,n+1)\concat \infty]_\beta = [\rho\concat(n+1)]_\beta$.  
 Furthermore, for each $m$, we know that 
 $\rho\concat(n,n+1)\concat(m,m+1)$ has rank strictly less than $\beta$, 
 and therefore $[\rho\concat(n,n+1)\concat m]_\beta=[\rho\concat(n,n+1)\concat(m+1)]_\beta$.  
So in fact $[\rho\concat n]_\beta$ and $[\rho\concat (n+1)]_\beta$ 
are successor classes, and the desired conclusion follows.

On the other hand, suppose that $r_T(\rho\concat(n,n+1)\geq \alpha$.  Then by fatness, 
there are infinitely many $m$ such that $r_T(\rho\concat(n,n+1)\concat(m,m+1)) = \beta$.
For such $m$, we have 
$[\rho\concat(n,n+1)\concat m]_\beta < [\rho\concat(n,n+1)\concat(m+1)]_\beta$.  Therefore, 
$[\rho\concat n]_\beta$ and $[\rho\concat (n+1)]_\beta$ are separated by infinitely 
many $L_T^\beta$ equivalence classes.

Letting $\alpha = r_T(\lambda)$, we see that $[0]_\alpha = [n]_\alpha$ for every $n$, 
so $[0]_{\alpha+1} = [\infty]_{\alpha+1}$, thus the Hausdorff rank is at most $\alpha+1$. 
But for any $\beta<\alpha$, there are infinitely many $m$ such that $[m]_\beta < [m+1]_\beta$, 
so $[0]_\alpha \neq [\infty]_\alpha$.  
\end{proof}

\begin{theorem}
For every ordinal $\alpha<\omega_1^{ck}$, there is a $\mathbb Z^2$-SFT 
with TCPE rank $\alpha$.
\end{theorem}
\begin{proof}
For any computable ordinal $\beta$, there is a computable well-founded fat tree $T\subseteq (2\Omega)^{<\omega}$ 
of rank $\beta$.  
Let $Y_T$ be the SFT defined in the previous section.  Then $Y_T$ has a fully
supported measure, and the generalized 
transitivity rank of $Y_T$ is either $2\beta-1$ or $2\beta$.  Therefore, 
letting $f$ be the ``put it on fabric'' operation of Theorem \ref{thm:pavlov-modified}, 
the SFT $f(Y_T)$ has TCPE rank $2\beta-1$ or $2\beta$.

But, we can say more, the TCPE rank of $Y_T$ is $2\beta -1$.
By the proof of the previous proposition, when $T$ is fat, $L_T^{\beta-1}$ 
has two equivalence classes, one containing type $\infty$, the other containing 
everything else.  Since every element of type $\infty$ is the limit of a sequence 
of elements of finite type, the topological closure step is sufficient to connect everything, 
and thus the TCPE rank is $2\beta-1$.  We can also get TCPE rank $2\beta$ 
from a fat tree $T$ of rank $\beta$ by considering $Y_S$, where 
$$S = \{(0,1)\concat \sigma : \sigma \in T\}.$$
Now $L_S^{\beta-1}$ has two equivalence classes, types less than 1 and 
types greater than or equal to 1.  In this case, the topological closure step is not sufficient 
to connect the 0 type to the $\infty$ type; the subsequent transitive 
and symmetric closure step is needed, so the rank of $Y_S$ is $2\beta$.  

Finally, we can achieve SFTs with limit TCPE rank, using a variation on
$Y_T$ for fat $T$, by connecting the 
$\infty$ type to the 0 type prematurely using some extra symbols. 
Adding symbols $\ast$ and $\dagger$ to $\Lambda$, we can add restrictions 
which ensure that if $\ast$ appears, then no gray symbol of $\Lambda$ appears, 
and if $\dagger$ appears, then no white symbol of $\Lambda$ appears. 
Then some elements of type $\infty$ are an entropy pair with $\ast^{\mathbb Z^2}$, 
$\ast^{\mathbb Z^2}$ is an entropy pair with $\dagger^{\mathbb Z^2}$, 
$\dagger^{\mathbb Z^2}$ is an entropy pair with some elements of type 0, 
and no other entropy pair relations are added.  These restrictions are also 
easily enforced with sofic computation.
\end{proof}

The same methods also allow us to show that TCPE admits no simpler description 
in the special case of $\mathbb Z^2$-SFTs.  It is not true that $Y_T$ has 
TCPE if and only if $T$ is well-founded, because some ill-founded trees still 
have a scattered $L_T$ (consider for example the tree with only a single path).
However, we can get around that by fattening the tree first.

\begin{theorem}
The property of TCPE is $\Pi^1_1$-complete in the set of $\mathbb Z^2$-SFTs.
\end{theorem}
\begin{proof}
We describe an algorithm which, given an index for a computable tree $S$, 
produces a $\mathbb Z^2$-SFT which has TCPE if and only if $S$ is well-founded. 
Uniformly in an index for $S$, we can produce an index for the tree 
$T\subseteq (2\Omega)^{<\omega}$ which 
is obtained by fattening $S$.  If $S$ was well-founded, then $T$ is also well-founded 
(the fattening does not increase the rank of any node).  But if $S$ was ill-founded, 
then $T$ is not only ill-founded, but also $L_T$ is not scattered, because 
$[T]$ contains a Cantor set, and so a copy of $\mathbb Q$ can be order-embedded into 
$L_T$ using types from $[T]$.  Therefore, $Y_T$ has TCPE if and only if $S$ 
was well-founded.
\end{proof}

\bibliographystyle{plain}
\bibliography{tcpe}

\end{document}